\documentclass[11pt]{amsart}
\usepackage{amssymb}
\usepackage{amscd}
\usepackage{color}
\usepackage{comment}
\usepackage[dvipsnames]{xcolor}
\usepackage{url}
\usepackage{hyperref}
\usepackage{tikz}
\usetikzlibrary{calc}
\usepackage{enumitem}

\newtheorem{theorem}{Theorem}[section]

\newtheorem{proposition}[theorem]{Proposition}
\newtheorem{corollary}[theorem]{Corollary}

\theoremstyle{definition}
\newtheorem{definition}[theorem]{Definition}
\newtheorem{example}[theorem]{Example}

\newtheorem*{definition*}{Definition}

\newenvironment{demode}
{\noindent {{\it Proof of }}}%
{\par \hfill \fbox{}}

\newtheorem*{theorem*}{Theorem}

\theoremstyle{remark}
\newtheorem{remark}[theorem]{Remark}

\numberwithin{equation}{section}

\newcommand{\norm}[1]{\left\|#1\right\|}

\DeclareMathOperator*{\Impart}{Im}

\newcommand{\ran}{\textnormal{ran}}

\newcommand{\supp}{\textnormal{supp}}
\newcommand{\pe}[1]{\langle#1\rangle}
\newcommand{\EL}{\mathcal{L}}

\newcommand{\C}{\mathbb{C}}

\newcommand{\N}{\mathbb{N}}
\newcommand{\D}{\mathbb{D}}
\newcommand{\T}{\mathbb{T}}
\newcommand{\PR}{\textnormal{Re}}
\newcommand{\F}{\mathcal{F}}

\newcommand{\al}{\alpha}

\newcommand{\Xint}{X_T(\overline{\inte(\gamma)})}
\newcommand{\Xext}{X_T(\overline{\exte(\gamma)})}
\newcommand{\Hint}{H_T(\overline{\inte(\gamma)})}
\newcommand{\Hext}{H_T(\overline{\exte(\gamma)})}
\newcommand{\sumn}{\sum_{n=1}^{\infty}}

\newcommand{\summ}{\sum_{m=1}^{\infty}}

\newcommand{\sumk}{\sum_{k=1}^{\infty}}

\newcommand{\y}{\textnormal{\textbf{y}}}
\newcommand{\ind}{\textnormal{ind}}

\newcommand{\R}{\mathbb{R}}

\newcommand{\x}{\textnormal{\textbf{x}} }
\newcommand{\sumi}{\sum_{i=1}^\infty}
\newcommand{\sumj}{\sum_{j=1}^\infty}
\newcommand {\conm}[1]{{\{#1\}'}}
\newcommand{\biconm}[1]{{\{#1\}''}}

\newcommand{\inte}{\textnormal{int}}
\newcommand{\exte}{\textnormal{ext}}

\begin{document}
	
	\title[Spectral cuts and unconventional functional calculi]{Spectral cuts and \\ unconventional functional calculi}

	\author{Eva A. Gallardo-Guti\'{e}rrez}
	\address{Eva A. Gallardo-Guti\'errez \newline
		Departamento de An\'alisis Matem\'atico y Matem\'atica Aplicada,\newline
		Facultad de Ciencias Matem\'aticas,
		\newline Universidad Complutense de
		Madrid, \newline
		Plaza de Ciencias N$^{\underbar{\Tiny o}}$ 3, 28040 Madrid,  Spain
		\newline
		and Instituto de Ciencias Matem\'aticas ICMAT (CSIC-UAM-UC3M-UCM),
		\newline Madrid,  Spain } \email{eva.gallardo@mat.ucm.es}

	\author[F. Javier González-Doña]{F. Javier González-Doña}
\address{F. Javier González-Doña
\newline  Departamento de Matemática Aplicada II and IMUS, \newline Escuela Técnica Superior de Ingeniería de Ingeniería,\newline
Universidad de Sevilla,\newline
Camino de los Descubrimientos s/n, 41092, Seville, Spain}
\email{fgonzalez13@us.es}

	\thanks{Both authors are partially supported by Plan Nacional  I+D grant no. PID2022-137294NB-I00, Spain,
		the Spanish Ministry of Science and Innovation, through the ``Severo Ochoa Programme for Centres of Excellence in R\&D'' CEX2023-
		001347 and from the Spanish National Research Council, through the ``Ayuda extraordinaria a Centros de Excelencia Severo Ochoa'' (20205CEX001).}
	
	
	\date{March 2026}
	

	\begin{abstract}
In this work, we prove that linear bounded operators $T$ on a Banach space $X$ allowing spectral cuts along rectifiable Jordan curves meeting their spectrum are related to classes of operators admitting an \emph{unconventional functional calculus}. We identify several such classes and address the consequences regarding the existence of non-trivial closed invariant subspaces, extending previous results of Chalendar \cite{Chalendar1996, Chalendar1997, Chalendar1998}. Furthermore, we establish that every operator belonging to a broad subclass of compact perturbations of diagonalizable normal operators on separable Hilbert spaces — namely, trace-class perturbations — possesses an unconventional functional calculus and is \emph{super-decomposable}, thereby extending earlier results obtained by the authors in \cite{GG3}.
\end{abstract}
	
	\maketitle

	\section{Introduction}

	Let $T$ be a bounded linear operator on an infinite-dimensional complex Banach space $X$. A classical theorem of Riesz states that, in case the spectrum of $T$ is disconnected, one can explicitly write an  invariant projection for $T$ as follows:
	$$
	P = \frac{1}{2\pi i} \int_{\gamma} (zI - T)^{-1} \, dz,
	$$
	where $I$ stands for the identity operator and $\gamma$ is a properly chosen loop. Clearly, the range of $P$ provides a non-trivial closed invariant subspace for $T$. Indeed, a hyperinvariant one (a subspace invariant under every operator in the commutant of $T$).
	
	\smallskip
	
	In the seventies, Stampfli \cite{Stampfli} obtained explicit results regarding this construction when the spectrum of $T$ is connected but the loop $\gamma $ can be chosen so that it does not intersect the spectrum too often, and $(zI - T)^{-1} $ is not too large near those intersections. In general, one does not obtain a projection, but non-zero operators $S_1, S_2$ in the double commutant of $T$ such that $S_1 S_2 = 0$, which suffices for the existence of non-trivial closed (hyper)-invariant subspaces.

	\smallskip
	
	In this respect, results on classes of operators that have a spectral decomposition of the underlying space have been closely related to the study of operators that have some sort of functional calculus, and these two aspects seem to be strongly related. The example of \emph{spectral operators} in the sense of Dunford \cite{Dunford and Schwarz} and
	their generalization, \emph{decomposable operators} introduced by Foia\c{s} \cite{FOIAS}, are probably the most illustrative cases in this sense.
	
	\smallskip
	
	The aim of this work is driven by such a relation, allowing us to identify classes of operators which admit an \emph{unconventional functional calculus}. Among those classes, we will find operators satisfying a suitable local resolvent growth condition along a curve $\gamma$  and compact perturbations of scalar-type spectral operators.
Various subclasses of such operators have been studied previously, both in classical works (\cite{LjubicMacaev1965}, \cite{Apostol1971}, \cite{Kitano1973}) and in more recent contributions (see \cite{Chalendar1996, Chalendar1997, Chalendar1998}, \cite{FX12}, \cite{AA}, \cite{PutinarYakubovich2021}, \cite{GG2, GG3}), in connection with the existence of non-trivial closed hyperinvariant subspaces and decomposability.

	\smallskip

Our starting point is the concept of a \emph{plain spectral cut} of $T$ along a rectifiable Jordan curve. In order to introduce it, recall that an operator $T$ has the single valued extension property (SVEP) if for every open set $U$ in the complex plane $\mathbb{C}$, the continuous linear mapping $T_U $ defined on the space of $X$-valued holomorphic functions $\mathcal{H}(U; X)$ by
	$$
	T_U f(z)= (zI - T) f(z) \quad  (f \in \mathcal{H}(U; X), \, z \in U),
	$$
	is injective. The \emph{local resolvent set} $\rho_T(x)$ at $x\in X$ consists  of all $z \in \mathbb{C}$ such that there exists an open neighborhood $U$ of $z$ and a function $f \in \mathcal{H}(U; X) $ satisfying
	$$
	(zI - T) f(z) \equiv x \quad \text{on } U.
	$$
	Its complement, $ \sigma_T(x) := \mathbb{C} \setminus \rho_T(x) $, is called \emph{the local spectrum of} $ T $ at $ x $. For sets $ \Omega \subseteq   \mathbb{C}$, the linear manifolds
	$$
	X_T(\Omega)= \{ x \in X : \sigma_T(x) \subseteq   \Omega \},
	$$
	are the \emph{spectral subspaces} of $T$ associated with $\Omega$. It is worth noticing that $X_T(\Omega)$ is a linear manifold that is hyperinvariant for $T$ but not necessarily closed even for closed subsets (see \cite[Chapter 2]{AIENA-book} for instance).

\medskip

Given linear manifolds $M$ and $N$ of $X$, we write $X=M\dotplus N$ to denote their algebraic direct sum, meaning that every $x\in X$ can be uniquely decomposed as $x=m+n$, with $m\in M$ and $n\in N$. Whenever $M$ and $N$ are closed subspaces, the previous sum becomes a topological direct sum due to the Open Mapping Theorem, and in such a case we denote it by $X=M\oplus N$.

\medskip

\begin{definition*} \emph{Let $T$ be a linear bounded operator on $X$ with the SVEP and $\gamma$ a rectifiable Jordan curve. The operator $T$ admits a \textbf{plain spectral cut along $\gamma$} if both spectral subspaces ${X_T(\overline{\inte(\gamma)})}$ and $X_T(\overline{\exte(\gamma)})$ are non-trivial closed subspaces and $X$ is the topological direct sum
\begin{equation}\label{direct sum}
X = X_T(\overline{\inte(\gamma)}) \oplus X_T(\overline{\exte(\gamma)}).
\end{equation}}
\end{definition*}
\medskip

\noindent Here, $\inte(\gamma)$ denotes the \emph{interior of $\gamma$}, namely, those complex numbers $z\in \mathbb{C}$ with index $\ind_\gamma(z)$ with respect to $\gamma$   equal to $\pm 1$, while $\exte(\gamma)$ stands for the \emph{exterior of $\gamma$} (zero index with respect to $\gamma$).

\medskip

Note that, if $T$ admits a plain spectral cut along a curve $\gamma$, then $X_T(\gamma) = \{0\}$. Also note that if $\tau$ is another rectifiable Jordan curve such that $\sigma(T)\cap \gamma = \sigma(T)\cap \tau\neq \varnothing$, then $\sigma(T)\cap \overline{\inte(\tau)}$ coincides either with $\sigma(T)\cap \overline{\exte(\gamma)}$ or with $\sigma(T)\cap \overline{\inte(\gamma)}$, and the same holds for $\sigma(T)\cap \overline{\exte(\tau)}$. Accordingly, $T$ also admits a plain spectral cut along $\tau$.

\smallskip

Before going further, a few words are in order. First, the assumption of non-triviality of both subspaces  $X_T(\overline{\inte(\gamma)})$ and $X_T(\overline{\exte(\gamma)})$  provides the perspective of interest since its absence may reduce the direct sum \eqref{direct sum} to known cases. Secondly,  Albrecht and Chevreau  introduced
in \cite[Definition 2.1]{AA} the concept of \emph{non-trivial clear spectral cut} along a rectifiable Jordan curve $\gamma$ for linear bounded operators $T$ on $X$ with the SVEP by imposing that $X$ is the topological direct sum
$$
X = \overline{X_T(\inte(\gamma))} \oplus \overline{X_T(\exte(\gamma))}
$$
and both  closed subspaces $\overline{X_T(\inte(\gamma))}$ and $\overline{X_T(\exte(\gamma))}$ are infinite dimensional. As we will show, the existence of plain spectral cuts for operators is further related to the existence of functional calculi.

\smallskip

Indeed, if $T$ is an operator with the SVEP, for every $x \in X$ there exists a unique holomorphic $X$-valued function $x_T : \rho_T(x) \to X$ such that
$$
	(zI - T) x_T(z) \equiv x \quad \text{on} \, \rho_T(x).
$$
	The function $x_T$ is called the \emph{local resolvent function} of $T$ at $x$.

\smallskip

Assume that $T$ admits a plain spectral cut along $\gamma$. Then, every $x$ belonging to the linear manifold $X_T(\inte(\gamma))\dotplus X_T(\exte(\gamma))$  can be expressed uniquely as $x = u + v $  with $u \in X_T(\inte(\gamma)) $ and $v \in X_T(\exte(\gamma))$. Moreover, $\sigma_T(x)\subseteq \sigma_T(u)\cup \sigma_T(v)$ \cite[Proposition 1.2.16]{LN00}. Observe that $\sigma_T(u)$ and $\sigma_T(v)$ do not intersect $\gamma$, so neither does $\sigma_T(x)$. As a consequence, the local resolvent functions $x_T(z), u_T(z)$ and $v_T(z)$ are well defined for all $z\in \gamma$, and the SVEP yields that $x_T(z) = u_T(z)+v_T(z)$. At this point, the Cauchy integral theorem along with the analytic functional calculus yield that
	\begin{eqnarray}\label{AA}
		\frac{1}{2\pi i} \int_\gamma x_T(z) \, dz &=& \frac{1}{2\pi i} \int_\gamma u_T(z) \, dz + \frac{1}{2\pi i} \int_\gamma v_T(z) \, dz \nonumber\\
		&= & \frac{1}{2\pi i} \int_{\{|z|=\|T\|+1\}} (zI- T)^{-1} u \, dz + 0 \nonumber\\
		&=& P_\gamma x
	\end{eqnarray}
where $P_\gamma$ stands for the corresponding continuous idempotent operator with range $X_T(\overline{\inte(\gamma)})$ and kernel $X_T(\overline{\exte(\gamma)})$ (see Proposition \ref{proposicion spectral idempotent}).
	
\smallskip

Equation \eqref{AA} was also observed in \cite[Remark 2.2]{AA} and will play a central role in what follows. It will allow us to introduce a consistent functional calculus for operators admitting plain spectral cuts along curves in Section~\ref{sec 2}, which we then extend to cycles in Section~\ref{sec 3}, where we also exhibit its connection with (super)decomposability.

\smallskip

In Section~\ref{sec 4} we show that operators satisfying a suitable local resolvent growth condition along a curve $\gamma$ admit plain spectral cuts, and that the associated unconventional functional calculus admits an explicit integral representation. Moreover, we prove that for a broad class of operators that may fail to admit a plain spectral cut, the same local resolvent growth condition still yields a weaker form of spectral cuts and, consequently, a corresponding unconventional functional calculus. As an application of these results, we obtain a criterion ensuring the existence of non-trivial closed hyperinvariant subspaces for the operator, extending previous results by Isabelle Chalendar \cite{Chalendar1996, Chalendar1997, Chalendar1998}.
	
\smallskip

In Section \ref{sec 5} we show that a broad subclass of compact perturbations of diagonalizable normal operators on separable Hilbert spaces—namely, trace-class perturbations—admits, for every operator in the class, an \emph{exhaustive cutting family of spectral cuts} (Definition \ref{definition exhaustive}). As a consequence, these operators possess an unconventional functional calculus and are \emph{super-decomposable}, extending earlier results obtained by the authors in \cite{GG3}.

\section{Plain spectral cuts and an unconventional functional calculus} \label{sec 2}

The main aim of this section is to introduce an unconventional functional calculus for operators admitting  plain spectral cuts along rectifiable Jordan curves inspired by the Borel functional calculus for normal operators in Hilbert spaces.

\smallskip

Throughout this section, $X$ will denote an infinite dimensional complex Banach space and $\mathcal{L}(X)$ the Banach algebra of the linear bounded operators acting on $X$. For each $T\in \mathcal{L}(X)$, $\sigma(T)$ will stand for the spectrum of $T$, namely, the set of all complex numbers $ z \in \mathbb{C} $ such that $ T -z I $ is not invertible in $X$.

\smallskip

The following  result characterizes operators $T$ admitting plain spectral cuts along $\gamma$ in terms of the existence of non-trivial idempotents in the bicommutant of $T$. Recall that the \emph{commutant} of $T$ is defined by
$\{T\}'= \{ A \in \mathcal{L}(X) : AT = TA \},$ and the \emph{double commutant} or \emph{bicommutant} of $T$ is the commutant of $\{T\}'$, namely
$$
\{T\}'' = \{ B \in \mathcal{L}(X) : BA = AB \text{ for all } A \in \{T\}' \}.
$$

\begin{proposition}\label{proposicion spectral idempotent}
		Let $T$ be a linear bounded operator on $X$ with the SVEP and  $\gamma$ a rectifiable Jordan curve. $T$ admits a plain spectral cut along $\gamma$ if and only if there exists a non-trivial idempotent $P_\gamma \in \biconm{T}$ such that
        $$\ran(P_\gamma) = X_T(\overline{\inte(\gamma)}), \qquad \ker(P_\gamma) =  X_T(\overline{\exte(\gamma)}).$$
In such a case, $\sigma(T\mid_{\ran(P_\gamma)})\subseteq   \overline{\inte(\gamma)}\cap\sigma(T)$ and $\sigma(T\mid_{\ker(P_\gamma)})\subseteq   \overline{\exte(\gamma)}\cap\sigma(T)$.
\end{proposition}

\begin{proof}
Assume that $T$ admits a plain spectral cut along $\gamma$. Equation \eqref{direct sum} implies that the bounded operator
$$P_\gamma = (I\mid_{X_T(\overline{\inte(\gamma)})}) \oplus(0\mid_{X_T(\overline{\exte(\gamma)})}),$$
belongs to $\biconm{T}$ since both $X_T(\overline{\inte(\gamma)})$ and $X_T(\overline{\exte(\gamma)})$ are hyperinvariant subspaces for $T$.

\smallskip

On the other hand, if such a non-trivial idempotent $P_\gamma$ exists, it is clear that \eqref{direct sum} holds, so $T$ admits a plain spectral cut. Finally, the statement for the spectrum of $T\mid_{\ran(P_\gamma)}$ and $T\mid_{\ker(P_\gamma)}$ follows directly by \cite[Proposition 1.2.20]{LN00}.
\end{proof}

\begin{remark}\label{remark 1}
It turns out that $P_\gamma$ is a \textit{spectral idempotent} associated with $\overline{\inte(\gamma)}$ in the sense of \cite{FJKP11} (see also \cite[Definition 3.1]{GG2}). An equivalent property holds for $Q_\gamma= I-P_\gamma$ and $\overline{\exte(\gamma)}$.
\end{remark}

\smallskip

The following result yields explicit spectral cuts under suitable conditions on the local spectral subspaces of $T$.

\smallskip

\begin{theorem}\label{teorema plain spectral cut local resolvent}
Let $T$ be a bounded linear operator on $X$ with the SVEP, and let $\gamma$ be a rectifiable Jordan curve such that both $\inte(\gamma)$ and $\exte(\gamma)$ intersect $\sigma(T)$. Suppose that $\Xint$ and $\Xext$ are closed, that the linear manifold $X_T(\inte(\gamma)) + X_T(\exte(\gamma))$ is dense in $X$, and that there exists $C>0$ such that
\begin{equation}\label{acotacion norma x_T}
    \norm{\frac{1}{2\pi i}\int_\gamma x_T(z)\,dz} \le C\norm{x}
\end{equation}
for all $x \in X_T(\inte(\gamma)) + X_T(\exte(\gamma))$. Then $T$ admits a non-trivial plain spectral cut along $\gamma$ if and only if $X_T(\gamma)=\{0\}$.
\end{theorem}

\smallskip

\begin{proof}
Recall that $\Xint \cap \Xext = X_T(\gamma)$. Hence, if $X_T(\gamma)\neq\{0\}$, then $T$ cannot admit a plain spectral cut along $\gamma$. Assume therefore that $X_T(\gamma)=\{0\}$ and let us show that $T$ admits the plain spectral cut. It suffices to prove that
$$
X = X_T(\overline{\inte(\gamma)}) + X_T(\overline{\exte(\gamma)}).
$$

For simplicity, set $M = X_T(\inte(\gamma)) + X_T(\exte(\gamma))$.  For each $x \in M$ there exist unique vectors $u \in X_T(\inte(\gamma))$ and $v \in X_T(\exte(\gamma))$ such that $x = u + v$. As noted in the discussion preceding \eqref{AA},
$$\sigma_T(x) \subseteq \sigma_T(u) \cup \sigma_T(v),$$
where $\sigma_T(u) \subset \inte(\gamma)$ and $\sigma_T(v) \subset \exte(\gamma)$ are compact. Hence $x_T(z)$, $u_T(z)$ and $v_T(z)$ are well defined for all $z \in \gamma$, and
$x_T = u_T + v_T.$

Define a linear operator $J_\gamma : M \to X_T(\inte(\gamma))$ by
\begin{equation}\label{operador J}
J_\gamma x = \frac{1}{2\pi i}\int_\gamma x_T(z)\,dz .
\end{equation}
By \eqref{AA}, $J_\gamma x = u$ and, consequently, $(I - J_\gamma)x = v \in X_T(\exte(\gamma)).$
It follows that
$$
J_\gamma(I - J_\gamma)x = (I - J_\gamma)J_\gamma x = 0,
$$
and therefore $J_\gamma$ is idempotent. Now, \eqref{acotacion norma x_T} and the density of $M$ imply that $J_\gamma$ extends to a bounded linear idempotent operator $J_\gamma : X \to X$. Since $X_T(\inte(\gamma)) \subseteq X_T(\overline{\inte(\gamma)})$ and the latter subspace is closed, we have
$$
\ran(J_\gamma) \subseteq \Xint,
\qquad
\ker(J_\gamma) \subseteq \Xext.
$$
Hence, for every $x \in X$,
$$
x = J_\gamma x + (I-J_\gamma)x \in \Xint + \Xext,
$$
which completes the proof.
\end{proof}

\begin{remark}A few remarks concerning Theorem \ref{teorema plain spectral cut local resolvent} are in order:
\begin{enumerate}
\item [(i)] A careful reading of the proof of Theorem \ref{teorema plain spectral cut local resolvent} shows that the idempotent operator $P_\gamma$ associated with the plain spectral cut is the operator $J_\gamma$ defined by \ref{operador J}. Consequently,
        $$
        P_\gamma x = \frac{1}{2\pi i}\int_\gamma x_T(z)dz
        $$
        for all $x\in X_T(\inte(\gamma))+X_T(\exte(\gamma))$.

\item [(ii)]The linear manifold $M = X_T(\inte(\gamma)) + X_T(\exte(\gamma))$ consists precisely of those vectors $x \in X$ for which $x_T(z)$ is well defined for all $z \in \gamma$. Indeed, this condition is equivalent to $\gamma \subseteq \rho_T(x)$, or equivalently $\gamma \cap \sigma_T(x)=\varnothing$. Hence
$$
\sigma_T(x) = K_1 \cup K_2,
$$
where $K_1 \subseteq \inte(\gamma)$ and $K_2 \subseteq \exte(\gamma)$ are disjoint compact sets. Consequently,
$$
x \in X_T(K_1 \cup K_2) = X_T(K_1)+X_T(K_2)
\subseteq X_T(\inte(\gamma))+X_T(\exte(\gamma)).
$$
The reverse inclusion was discussed prior to \eqref{AA}.

\item [(iii)] The hypotheses in Theorem \ref{teorema plain spectral cut local resolvent} require a priori information on the linear manifold $X_T(\inte(\gamma)) + X_T(\exte(\gamma))$. In Section \ref{sec 4}, we consider similar constructions of plain spectral cuts, replacing these assumptions with certain conditions on $\sigma(T)$ and the resolvent $(z-T)^{-1}$ along $\gamma$.
    \end{enumerate}
\end{remark}

\noindent The next definition introduces $f_\gamma(T)$  when $T$ admits a plain spectral cut along $\gamma$ and $f$ is holomorphic in $\overline{\inte(\gamma)}$.

\smallskip
	
\begin{definition}\label{definicion unconventional functional calculus}
		Let $T$ be a linear bounded operator on $X$ with the SVEP and $\gamma$ a rectifiable Jordan curve. Assume that $T$ admits a plain spectral cut along $\gamma$. If $U\subseteq  \C$ is an open subset containing $\overline{\inte(\gamma)}$ and $f: U \rightarrow \C$ is a holomorphic function, the operator $f_\gamma(T)$ is defined by
		\begin{equation}\label{expresion f_Gamma}
			f_\gamma(T) = f(T\mid_{\ran (P_\gamma)})P_\gamma,
		\end{equation}
		where $f(T\mid_{\ran (P_\gamma)})$ is determined by the Dunford functional calculus.
	\end{definition}

\smallskip
	
The next result shows that Definition \ref{definicion unconventional functional calculus} is consistent.

\smallskip
	
\begin{theorem}\label{teorema calculo bien definido}
		Let $T$ be a linear bounded operator on $X$ with the SVEP, $\gamma$ a rectifiable Jordan curve and assume that $T$ admits a plain spectral cut along $\gamma$. Let $U\subseteq  \C$ be an open subset containing $\overline{\inte(\gamma)}$ and $f: U \rightarrow \C$ be a holomorphic function. Then, $f_\gamma(T)$ is a linear bounded operator acting on $X$ with range $\ran(f_\gamma(T))\subseteq   {X_T(\overline{\inte(\gamma)})}$. Moreover, if $U$ contains $\sigma(T)$, then
		$$
		f_\gamma(T) = f(T)P_\gamma.
		$$
\end{theorem}

\begin{proof}
		First, let us show that  $f_\gamma(T)$ is well defined. Since $\ran(P_\gamma) = {X_T(\overline{\inte(\gamma)})}$, by Proposition \ref{proposicion spectral idempotent} $\sigma(T\mid_{\ran(P_\gamma)})$ is a compact subset contained in $\overline{\inte(\gamma)}\cap\sigma(T).$ Thus, $U$ is an open subset containing \mbox{$\sigma(T\mid_{\ran(P_\gamma)})$} and therefore, $f(T\mid_{\ran(P_\gamma)})$ is a linear bounded operator acting on $\ran(P_\gamma)$.
As a consequence,
$$f_\gamma(T)= f(T\mid_{\ran(P_\gamma)})P_\gamma$$
is a linear bounded operator on $X$ with range contained in $\ran(P_\gamma) = X_T(\overline{\inte(\gamma)})$.

\smallskip		

Finally, if $U$ contains $\sigma(T)$, clearly $f(T) \in \EL(X)$ and
		$$f(T\mid_{\ran(P_\gamma)})x = f(T)x$$
		for all $x\in \ran(P_\gamma)$, which concludes the proof.
	\end{proof}

Theorem \ref{teorema calculo bien definido} allows us to consider a \emph{functional calculus for $T$ associated with $\gamma$}, and the next result states that is related to \eqref{AA}:
	
\begin{theorem}\label{teorema expresion integral calculo funcional inconvencional}
		Let $T$ be a linear bounded operator on $X$ with the SVEP, $\gamma$ a rectifiable Jordan curve and assume that $T$ admits a plain spectral cut along $\gamma$. Let $U\subseteq  \C$ be an open subset containing $\overline{\inte(\gamma)}$ and $f: U \rightarrow \C$ be a holomorphic function. For every $x\in X_T(\inte(\gamma))\dotplus X_T(\exte(\gamma))$
		\begin{equation}\label{expresion integral calculo funcional inconvencional}
			f_\gamma(T) x = \frac{1}{2\pi i}\int_\gamma f(z) x_T(z)dz.
		\end{equation}
	\end{theorem}

\smallskip

	\begin{proof}
Let $x\in X_T(\inte(\gamma))\dotplus X_T(\exte(\gamma))$ and note that $x= P_\gamma x + Q_\gamma x$, with $P_\gamma x\in X_T(\inte(\gamma))$ and $Q_\gamma x = (I-P_\gamma)x \in X_T(\exte(\gamma))$. As it was discussed regarding \eqref{AA}, $\gamma\subset \rho_T(x)$, so $x_T(z)$ is well defined for all $z\in \gamma$ and $x_T(z) = (P_\gamma x)_T(z) + (Q_\gamma x)_T(z)$. Consequently,
        $$
         \frac{1}{2\pi i}\int_\gamma f(z) x_T(z)dz= \frac{1}{2\pi i} \int_\gamma f(z) (P_\gamma x)_T(z) dz + \frac{1}{2\pi i} \int_\gamma f(z) (Q_\gamma x)_T(z) dz.
        $$
Let us show that
    $$
    \frac{1}{2\pi i} \int_\gamma f(z) (P_\gamma x)_T(z) dz = f_\gamma(T)x, \qquad \frac{1}{2\pi i} \int_\gamma f(z) (Q_\gamma x)_T(z) dz =0,
    $$
    and the statement of the theorem would follow.

Now, $\gamma \subseteq  \rho_T(P_\gamma x)$  for all $z\in \gamma$, so
$$(P_\gamma x)_T(z) = (T\mid_{\ran(P_\gamma)}-zI)^{-1}P_\gamma x.$$
Consequently,
		\begin{equation*}
			\begin{split}
				\frac{1}{2\pi i}\int_\gamma f(z) (P_\gamma x)_T(z) dz = \frac{1}{2\pi i} \int_\gamma f(z) (zI-T\mid_{\ran(P_\gamma)})^{-1}P_\gamma x dz.
			\end{split}
		\end{equation*}
Since $\gamma$ is a rectifiable Jordan curve surrounding $\sigma(T\mid_{\ran(P_\gamma)})$ and $f$ is holomorphic in an open set containing $\gamma$,
$$\frac{1}{2\pi i} \int_\gamma f(z) (zI-T\mid_{\ran(P_\gamma)})^{-1}P_\gamma x dz = f(T\mid_{\ran(P_\gamma)}) P_\gamma x = f_\gamma(T)x,$$
as claimed.

\smallskip

Finally,  recall that $\sigma_T(Q_\gamma x)$ is a closed subset contained in $\exte(\gamma)$. This means that $\gamma$ does not surround or intersect $\sigma_T(Q_\gamma x)$. Likewise,
$$(Q_\gamma x)_T(z) = (zI-T\mid_{\ran(Q_\gamma)})^{-1}Q_\gamma x$$
for all $z\in \gamma,$ so by Cauchy's Theorem
$$\frac{1}{2\pi i} \int_\gamma f(z)(Q_\gamma x)_T(z)dz = \frac{1}{2\pi i}\int_\gamma f(z) (T\mid_{\ran(Q_\gamma)}-zI)^{-1}Q_\gamma x = 0, $$
which concludes the proof.
\end{proof}

\smallskip

Next, we show that the constructed functional calculus associated with a curve satisfies the standard properties of a functional calculus:

\smallskip
	
\begin{theorem}\label{Teorema CF}
		Let $T$ be a linear bounded operator on $X$ with the SVEP, $\gamma$ be a rectifiable Jordan curve and assume that $T$ admits a plain spectral cut along $\gamma$. Let $U\subseteq  \C$ be an open subset containing $\overline{\inte(\gamma)}$ and $f, g$ be two holomorphic functions on $U$. Then:
\begin{enumerate}
			\item [\rm{(i)}] $1_\gamma(T) = P_\gamma$.
			\item [\rm{(ii)}] $(f+g)_\gamma(T) = f_\gamma(T)+g_\gamma(T).$
			\item [\rm{(iii)}] $(fg)_\gamma(T) = f_\gamma(T)g_\gamma(T). $
			\item [\rm{(iv)}] $\sigma(f_\gamma (T)) = f(\sigma(T\mid_{\ran(P_\gamma)}))\cup \{0\}\subseteq   f(\overline{\inte(\gamma)})\cup \{0\}.$
			\item [\rm{(v)}] If $(f_n)_{n\in \N}$ is a sequence of holomorphic functions on $U$ converging uniformly on compact subsets to $f$, then $(f_n)_\gamma(T) \rightarrow f_\gamma(T)$ in the operator norm.
			\item [\rm{(vi)}] If $\tilde{\gamma}$ is a rectifiable Jordan curve such that $T$ admits a non-trivial closed spectral cut through $\tilde{\gamma}$, $\inte(\gamma)\cap\inte(\tilde{\gamma}) \neq \varnothing$ and $\gamma \cap \sigma(T) = \tilde{\gamma}\cap \sigma(T)$,  then $f_\gamma(T) = f_{\tilde{\gamma}}(T).$
		\end{enumerate}
	\end{theorem}

\smallskip

	\begin{proof}
The proofs for (i), (ii) and (v) follow straightforwardly from \eqref{expresion f_Gamma}. To prove (iii), observe that
$$(fg)(T\mid_{\ran(P_\gamma)}) = f(T\mid_{\ran(P_\gamma)})g(T\mid_{\ran(P_\gamma)}).$$
Since both operators $f(T\mid_{\ran(P_\gamma)})$ and $g(T\mid_{\ran(P_\gamma)})$  act on $\ran(P_\gamma)$, it holds that
		$$
		f(T\mid_{\ran(P_\gamma)})g(T\mid_{\ran(P_\gamma)}) = f(T\mid_{\ran(P_\gamma)})P_\gamma \, g(T\mid_{\ran(P_\gamma)}).
		$$
		Then,
		$$
		(fg)_\gamma(T) = f(T\mid_{\ran(P_\gamma)})g(T\mid_{\ran(P_\gamma)})P_\gamma = f(T\mid_{\ran(P_\gamma)})P_\gamma g(T\mid_{\ran(P_\gamma)})P_\gamma = f_\gamma(T)g_\gamma(T),
		$$
		and (iii) follows. In order to show (iv), note that  $X= \ran(P_\gamma)\oplus \ran(Q_\gamma)$ and hence
		$$f_\gamma(T) = f(T\mid_{\ran(P_\gamma)})\oplus 0.$$
By \cite[Proposition 0.3, p. 9]{RR}, it follows that
		$$\sigma(f_\gamma(T)) = \sigma(f(T\mid_{\ran(P_\gamma)}))\cup \{0\}.$$
This along with the Spectral Mapping Theorem yields that
		$$
		\sigma(f(T\mid_{\ran(P_\gamma)})) = f(\sigma(T\mid_{\ran(P_\gamma)})),
		$$
		which shows that $\sigma(f_\gamma(T))= f(\sigma(T\mid_{\ran(P_\gamma)}))\cup\{0\}$.  In addition,  by Proposition \ref{proposicion spectral idempotent}, it follows that $  \sigma(T\mid_{\ran(P_\gamma)})\subseteq   \overline{\inte(\gamma)}$, and the desired inclusion follows.
		
\smallskip

Finally, let us show (vi). For this purpose, it suffices to show that $P_\gamma = P_{\tilde{\gamma}}$. Observe that, by hypotheses, $\overline{\inte(\gamma)} \cap \sigma(T) = \overline{\inte(\tilde{\gamma})}\cap \sigma(T)$ and $\overline{\exte(\gamma)} \cap \sigma(T) = \overline{\exte({\tilde{\gamma}})}\cap \sigma(T)$, so elementary properties of local spectral subspaces yield that
$$X_T(\overline{\inte(\gamma)}) = X_T(\overline{\inte(\tilde{\gamma})}) \mbox{ and } \Xext = X_T(\overline{\exte(\tilde{\gamma})}).$$
Hence, $P_\gamma = P_{\tilde{\gamma}},$ which proves (vi), and therefore, Theorem \ref{Teorema CF}.
\end{proof}

We close the section with the following remark regarding approaches to construct functional calculus through the local spectral theory.

\begin{remark}
A set $D$ in the complex plane is called a \emph{Cauchy domain} if it is open, it has a finite number of components, and the boundary of $D$ is composed of
a finite number of simple closed rectifiable curves, no two of which intersect.

In \cite{Bermudez1997}, Berm{\'u}dez, Gonz{\'a}lez and Martin{\'o}n  introduced a functional calculus as follows: for any holomorphic
function $f$ on a domain $\Delta(f)$ and any $T\in \EL(X)$ satisfying the
SVEP, $f[T]:D(f[T])\subset X\to X$, with domain
$$
D(f[T])= \{x\in X : \sigma_T(x)\subset \Delta(f)\}
$$
and $f[T]x$ given by
$$
f[T]x=\frac{1}{2\pi i}\int_{\Gamma} f(\lambda)x_T(\lambda)\,d\lambda,
$$
where $\Gamma$ is the boundary of any Cauchy domain $D$ such that
$\sigma_T(x)\subset D\subset \overline{D}\subset \Delta(f)$.

It holds that $D(f[T])$ is a linear subspace invariant under $T$. Moreover, in the case $\sigma(T)\subset \Delta(f)$, clearly $f[T]$
coincides with $f(T)$, the operator of the holomorphic functional calculus. The authors obtained a local spectral mapping theorem, according to which
$$
f(\sigma_T(x)) = \sigma_{f[T]}(x),
$$
and, as a consequence the stability of SVEP under local functional calculus is proved (see \cite{BermudezGonzalezMartinon2002}).
\end{remark}

\section{Plain spectral cuts along cycles and decomposability}\label{sec 3}

In this section, we extend the notion of plain spectral cuts to cycles of curves. This generalization enables us to consider direct sum decompositions of $X$ as in \eqref{direct sum}, where the associated spectral subspaces correspond to finite unions of $n$-connected  domains rather than only simply connected domains. Working in this broader framework allows us to explore the relationship between plain spectral cuts and the notions of decomposability and super-decomposability.

Following \cite[Chapter 13]{Conway2}, a cycle $\Gamma=\{\gamma_1,\cdots,\gamma_n\}$ of disjoint rectifiable Jordan curves is said to be \textit{positive} if, for every point $a\in \C \setminus \left(\cup_{k=1}^n \gamma_k\right)$, its index
$$
\ind_\Gamma(a)= \sum_{k=1}^n \ind_{\gamma_k}(a)
$$
takes only the values $0$ or $1$.

In this case, a positive cycle $\Gamma$ admits two possible orientations. To fix a canonical orientation for the positive cycles considered throughout the paper, we say that a positive cycle $\Gamma=\{\gamma_1,\cdots,\gamma_n\}$ is \textit{admissible} if each curve $\gamma_k$ has
$$
\left\{ \begin{matrix}
     \textnormal{positive orientation}, & \textnormal{if } d(\gamma_k) \textnormal{ is even}, \\
     \textnormal{negative orientation}, & \textnormal{if } d(\gamma_k) \textnormal{ is odd},
\end{matrix}\right.
$$
where $d(\gamma_k)$ denotes the number of curves $\gamma_i\in\Gamma$ contained in $\inte(\gamma_k)$, that is,
$$
d(\gamma_k)=\#\{\gamma_i\in\Gamma:\gamma_i\subseteq \inte(\gamma_k)\}.
$$
In this situation, the interior and exterior of $\Gamma$ are the open sets
$$
\inte(\Gamma)=\{a\in\C:\ind_\Gamma(a)=1\}
\quad\text{and}\quad
\exte(\Gamma)=\{a\in\C:\ind_\Gamma(a)=0\},
$$
respectively.

\begin{definition}
Let $T$ be a linear bounded operator on $X$ with the SVEP and $\Gamma = \{\gamma_1,\cdots,\gamma_n\}$ an admissible cycle. We say that $T$ admits a plain spectral cut along $\Gamma$ if both $X_T(\overline{\inte(\Gamma)})$ and $X_T(\overline{\exte(\Gamma)})$ are non-trivial closed subspaces and $X$ is the topological direct sum
    $$
    X= X_T(\overline{\inte(\Gamma)})\oplus X_T(\overline{\exte(\Gamma)}).
    $$
\end{definition}
Needless to say, all the results presented in Section \ref{sec 2} also hold for operators that admit a plain spectral cut along an admissible cycle $\Gamma$, rather than along a single curve $\gamma$.

\noindent In addition, statement (vi) in Theorem \ref{Teorema CF} allows us to introduce the concept of \emph{$T$-spectrally equivalent cycles}:

\smallskip

\begin{definition}
Let $T$ be a linear bounded operator on $X$ with the SVEP and $\Gamma, \tilde{\Gamma}$ two admissible cycles. The cycles $\Gamma$ and $\tilde{\Gamma}$ are $T$-spectrally equivalent, $\Gamma \sim_{\sigma(T)} \tilde{\Gamma}$, if $T$ admits plain spectral cuts along both $\Gamma$ and $\tilde{\Gamma}$, $X_T(\overline{\inte(\Gamma)})=X_T(\overline{\inte(\tilde{\Gamma})})$ and $\Gamma \cap \sigma(T) = \tilde{\Gamma}\cap \sigma(T)$.
\end{definition}

\smallskip

Our next goal is to show that the existence of a sufficiently rich family of $T$-spectrally equivalent cycles  for a bounded linear operator $T$ is closely related to decomposability.
Recall that a linear bounded  operator $T$ on $X$ is decomposable if for every open cover $\{U,V\}$ of $\mathbb{C}$ there exist two closed invariant subspaces $X_1,X_2 \subseteq
X$ such that
$$
	\sigma(T\vert_{X_1}) \subseteq   \overline{U} \, \text{ and } \, \sigma(T\vert_{X_2}) \subseteq   \overline{V},
$$
and $X= X_1+X_2$. Note that, in general, the sum decomposition is not direct, and the spectra of the restrictions need not be disjoint. Decomposable operators were introduced by Foia\c{s} \cite{FOIAS} in the 1960s as a generalization of spectral operators in the sense of Dunford \cite{Dunford and Schwarz}. Foia\c{s}'s original definition was somewhat more technical, but equivalent to the one given here (see \cite{LN00} for further details).

\smallskip

Recall that $T\in \EL(X)$ is \textit{super-decomposable} if for every open cover $\{U,V\}$ of $\sigma(T)$, there exists a linear bounded operator $R$ on $X$ commuting with $T$ such that
$$\sigma(T\mid{\overline{RX}})\subseteq  \overline{U}\qquad \sigma(T\mid{\overline{(I-R)X}})\subseteq  \overline{V}. $$
Clearly, every super-decomposable operator is decomposable.

\medskip

Plain spectral cuts and decomposability are connected through the following concept:

\begin{definition}\label{definition exhaustive}
     Let $T$ be a linear bounded operator on $X$ with the SVEP and  $\mathcal{F}= (\Gamma_i)_{i\in\mathfrak{I}}$ a family of cycles of disjoint rectifiable Jordan curves. The family $\mathcal{F}$ is an \textit{exhaustive cutting family} for $T$ if:
     \begin{enumerate}
          \item [(i)]$T$ admits a plain spectral cut along $\Gamma_i$ for all $i\in \mathfrak{I}$.
        \item [(ii)] For all non-empty open sets $U,V\subset\C$ such that $\sigma(T)\subseteq   U\cup V$, there exists $i\in \mathfrak{I}$ such that $\overline{\inte(\Gamma_i)}\cap \sigma(T) \subseteq  \overline{U}$ and $\overline{\exte(\Gamma_i)}\cap \sigma(T) \subseteq  \overline{V}$.
     \end{enumerate}
\end{definition}

With the definition of an exhaustive cutting family for $T$, the following result follows straightforwardly:

\begin{theorem}\label{teorema super-decomposable}
Every linear bounded operator $T$ acting on $X$ with the SVEP and admitting an exhaustive cutting family $\mathcal{F}= (\Gamma_i)_{i\in\mathfrak{I}}$ is super-decomposable.
\end{theorem}

\smallskip

\begin{proof}  Let $U,V\subseteq  \C$ be non-empty open sets such that $\sigma(T)\subseteq  U\cup V$. Let $\Gamma_{i_0}\in \mathcal{F}$ be a cycle such that $T$ admits a plain spectral cut along it, satisfying  $\overline{\inte(\Gamma_{i_0})}\cap \sigma(T) \subseteq  \overline{U}$ and $\overline{\exte(\Gamma_{i_0})}\cap \sigma(T) \subseteq  \overline{V}$.

Since $X=X_T(\overline{\inte(\Gamma_{i_0})})\oplus X_T(\overline{\exte(\Gamma_{i_0})})$, there exists an idempotent operator $P_{{i_0}} \in \biconm{T}$ such that $\ran(P_{{i_0}}) =X_T(\overline{\inte(\Gamma_{i_0})})$ and $\ran(I-P_{{i_0}}) =X_T(\overline{\exte(\Gamma_{i_0})}).$ Observe that
$$\sigma(T\mid_{\ran(P_{{i_0}})}) = \sigma(T\mid_{X_T(\overline{\inte(\Gamma_{i_0})})})   \subseteq   \overline{\inte(\Gamma_{i_0})}\cap \sigma(T) \subseteq   \overline{U}.$$
 Similarly, $\sigma(T\mid_{\ran(I-P_{{i_0}})}) \subseteq  \overline{V},$ and the statement of the theorem follows.
\end{proof}

To provide non-trivial examples of operators admitting an exhaustive cutting family, we state the following result, which allows the construction of new plain spectral cuts from known admissible cycles for the operator $T$.

\begin{proposition}\label{proposicion union curvas}
Let $T$ be a linear bounded operator on $X$ with the SVEP and $\Gamma_1,\Gamma_2$ two admissible cycles. Assume that $T$ admits plain spectral cuts along both $\Gamma_1$ and $\Gamma_2$ with associated projections $P_{\Gamma_1}$ and $P_{\Gamma_2}$, respectively. Assume further that $\overline{\inte(\Gamma_1)}\cap \overline{\inte(\Gamma_2)}=\varnothing.$ Then:
    \begin{enumerate}
        \item [(i)] The cycle $\Gamma = \Gamma_1\cup \Gamma_2$ can be oriented to form an admissible cycle satisfying $\inte(\Gamma) = \inte(\Gamma_1)\cup \inte(\Gamma_2).$
        \item [(ii)] $T$ admits a plain spectral cut along $\Gamma$, with associated projection $P_\Gamma = P_{\Gamma_1}+P_{\Gamma_2}$.
    \end{enumerate}
\end{proposition}

\begin{proof}
The assumption $\overline{\inte(\Gamma_1)}\cap \overline{\inte(\Gamma_2)}=\varnothing$ implies that $\Gamma = \Gamma_1\cup \Gamma_2$ is a cycle of disjoint rectifiable Jordan curves. Moreover, it is straightforward to verify that $\Gamma$ is admissible with an orientation  such that
$$
\mathrm{int}(\Gamma)=\mathrm{int}(\Gamma_1)\cup \mathrm{int}(\Gamma_2).
$$
Hence, $(i)$ follows.

\smallskip

To prove $(ii)$, let us show that the operator $P_\Gamma = P_{\Gamma_1}+P_{\Gamma_2}$, which clearly belongs to $\biconm{T}$, is an idempotent operator. For such a purpose, it suffices to prove that $P_{\Gamma_1}{P_{\Gamma_2}}= 0$. Note that
$$
\ran(P_{\Gamma_1}P_{\Gamma_2}) \subseteq   \ran(P_{\Gamma_1})\cap\ran(P_{\Gamma_2}) \subseteq   X_T(\overline{\inte(\Gamma_1)})\cap X_T(\overline{\inte(\Gamma_2)}) = X_T(\varnothing) = \{0\}.
$$
Now, we show that $\ran(P_\Gamma) =  X_T(\overline{\inte(\Gamma)})$ and $\ker(P_\Gamma) = X_T(\overline{\exte(\Gamma)})$, which will finish the proof.
By \cite[Proposition 1.2.16]{LN00}, we have
 $$
 X_T(\overline{\inte(\Gamma)}) = X_T(\overline{\inte(\Gamma_1)})+X_T(\overline{\inte(\Gamma_2)}).
 $$
 To prove the first equality, observe that
$$
\ran(P_\Gamma) = \ran(P_{\Gamma_1}+P_{\Gamma_2}) \subseteq   \ran(P_{\Gamma_1})+\ran(P_{\Gamma_2}) = X_T(\overline{\inte(\Gamma)}).
$$
For the reverse inclusion, let $z\in X_T(\overline{\inte(\Gamma)})$ and write $z=x+y$ with $x\in X_T(\overline{\inte(\Gamma_1)})=\ran(P_{\Gamma_1})$ and $y\in X_T(\overline{\inte(\Gamma_2)})=\ran(P_{\Gamma_2})$. Then
$$P_\Gamma z = (P_{\Gamma_1}+P_{\Gamma_2})(x+y) = x+y = z,$$ so $z\in \ran(P_\Gamma)$. It remains to show that $\ker(P_\Gamma) = X_T(\overline{\exte(\Gamma)})$. Since $P_{\Gamma_1}{P_{\Gamma_2}}= 0$, it is straightforward to check that $\ker(P_{\Gamma_1})\cap\ker(P_{\Gamma_2}) = \ker(P_{\Gamma_1}+P_{\Gamma_2})$. Now,
$$\ker(P_\Gamma)=\ker(P_{\Gamma_1}+P_{\Gamma_2}) = \ker(P_{\Gamma_1})\cap\ker(P_{\Gamma_2}) = X_T(\overline{\exte(\Gamma_1)})\cap X_T(\overline{\exte(\Gamma_2)}) = X_T(\overline{\exte(\Gamma)}), $$
which ends the proof.
\end{proof}


Next, we illustrate how an exhaustive cutting family can be constructed from a given family.

\begin{example}\label{ejemplo 1}
Let $A\subseteq \mathbb{R}$ be a dense subset of $\mathbb{R}$ and let $T$ be a bounded linear operator with the SVEP that admits plain spectral cuts along all rectifiable Jordan curves $\gamma$ satisfying the following properties:
\begin{enumerate}
    \item[(i)] $\gamma$ is a polygonal curve whose segments are parallel to either the real or the imaginary axis.
    \item[(ii)] The endpoints of the segments belong to $A\times A\subseteq \mathbb{C}$.
    \item[(iii)] Both $\mathrm{int}(\gamma)$ and $\mathrm{ext}(\gamma)$ intersect $\sigma(T)$.
\end{enumerate}

By Proposition \ref{proposicion union curvas}, we can generate a family $\mathcal{F}$ of admissible cycles $\Gamma=\{\gamma_1,\ldots,\gamma_n\}$ along which $T$ admits  plain spectral cuts and where all $\gamma_1,\ldots,\gamma_n$ satisfy the properties listed above.

\smallskip

Let us prove that the family $\mathcal{F}$ is an exhaustive cutting family for $T$. This idea is inspired by \cite[Theorem 3.2]{FJKP11}.

\smallskip

Let $U$ and $V$ be non-empty open subsets of $\C$ such that $\sigma(T)\subseteq   U\cup V$. We will show that there exists $\Gamma\in \mathcal{F}$ satisfying $\overline{\inte(\Gamma)}\cap \sigma(T) \subseteq  \overline{U}$ and $\overline{\exte(\Gamma)}\cap \sigma(T) \subseteq  \overline{V}$.  For such a purpose, consider open subsets $G_1,G_2\subseteq  \C$ such that $\overline{G_1}\subseteq   U$, $\overline{G_2}\subseteq   V$ and $\sigma(T)\subseteq   G_1\cup G_2$.
Define
$$
\delta = \min \{\textnormal{dist}(\overline{G_1},\C\setminus U), \textnormal{dist}(\overline{G_2},\C\setminus V)\}>0.
$$
Now consider a rectangle $[a,b]\times [c,d]\subseteq \mathbb{C}$ with $a,b,c,d\in A$ containing $\sigma(T)$. Construct a grid on $[a,b]\times [c,d]$ such that the intermediate points
$$
a=x_1<x_2<\cdots<x_n=b \quad \text{and} \quad c=y_1<y_2<\cdots<y_m=d
$$
belong to $A$, and the mesh of each partition is smaller than $\delta/4$.

Define
$$
\sigma=\bigcup_{([x_k,x_{k+1}]\times [y_\ell,y_{\ell+1}])\cap \overline{G_1}\neq \varnothing}
[x_k,x_{k+1}]\times [y_\ell,y_{\ell+1}].
$$
Observe that there exists an admissible cycle $\Gamma$ such that
$$
\sigma=\overline{\mathrm{int}(\Gamma)},
$$
since $\partial\sigma$ is a finite union of polygonal curves whose segments are parallel to the coordinate axes and whose endpoints lie in $A\times A$.

Moreover, by the choice of $\delta$,
$$
\overline{G_1}\subseteq \overline{\mathrm{int}(\Gamma)} \subseteq U.
$$
It remains to show that
$$
\overline{\mathrm{ext}(\Gamma)}\cap \sigma(T)\subseteq V.
$$
Let $x\in \overline{\mathrm{ext}(\Gamma)}\cap \sigma(T)$. Observe that any rectangle $[x_k,x_{k+1}]\times [y_\ell,y_{\ell+1}]$ of the grid containing $x$ does not intersect $\overline{G_1}$. Since $x\in \sigma(T)$ and $\sigma(T)\subseteq G_1\cup G_2$, it follows that $x\in G_2\subseteq V$. Therefore, $\mathcal{F}$ is an exhaustive cutting family, as desired.
\end{example}

To conclude this section, we will show that admitting a simpler family of plain spectral cuts that does not form an exhaustive cutting family is, nevertheless, sufficient to establish the decomposability of the operator.

To this end, we introduce two additional results that allow us to construct non-trivial spectral cuts (or weaker versions of them) from known ones. The first result extends, in some sense, Proposition \ref{proposicion union curvas} to curves that may intersect; however, the direct sum decomposition obtained does not arise from a plain spectral cut.

\smallskip

\begin{proposition}\label{proposicion union curvas que cortan}
Let $T$ be a bounded linear operator on $X$ with the SVEP, and let $\gamma_1,\ldots,\gamma_n$ be positively oriented rectifiable Jordan curves such that $T$ admits a plain spectral cut along each $\gamma_k$, with associated projections $P_{\gamma_k}$. Assume that
$\mathrm{int}(\gamma_i)\cap \mathrm{int}(\gamma_j)=\varnothing \text{ for } i\neq j,
$ and denote $\gamma_{i,j}=\gamma_i\cap\gamma_j$. If
$$
\beta=\overline{(\bigcup\limits_{k=1}^n \gamma_k)\setminus \left(\bigcup_{i,j=1, \ i\neq j}^n \gamma_{i,j} \right) },
$$
then
$$
X = X_T\!\left(\overline{\mathrm{int}(\beta)}\right)
\oplus
X_T\!\left(\overline{\mathrm{ext}(\beta)} \cup
\left(\bigcup_{i\neq j}\gamma_{i,j}\right)\right),
$$
and the associated projection is $P_{\beta}=\sum_{k=1}^n P_{\gamma_k}$.
\end{proposition}

\begin{proof}
Proceeding as in the proof of Proposition \ref{proposicion union curvas}, we first show that $P_\beta \in \biconm{T}$ is an idempotent. It suffices to show that $P_{\gamma_i}P_{\gamma_j}=0$ for all $i\neq j$. Note that
$$\ran(P_{\gamma_i}P_{\gamma_j}) \subseteq   X_T(\overline{\inte(\gamma_i)})\cap X_T(\overline{\inte(\gamma_j)}) = X_T(\gamma_i\cap \gamma_j) \subseteq   X_T(\gamma_i) = \{0\},$$ so $P_\beta$ is an idempotent operator.

\smallskip

Now, let us prove that $\ran(P_\beta) = X_T(\overline{\inte(\beta)})$ and $\ker(P_\beta) = X_T\left(\overline{\exte(\beta)}\cup \left(\cup_{i,j=1, \ i\neq j}^n \gamma_{i,j} \right) \right)$, which will yield the statement. For the range identity, note that
    $$
    X_T(\overline{\inte(\gamma_1)})+\cdots + X_T(\overline{\inte(\gamma_n)})\subseteq   X_T(\overline{\inte(\beta)}),
    $$ so we deduce that $\ran(P_\beta)\subseteq   X_T(\overline{\inte(\beta)})$. Now, let $x\in X_T(\overline{\inte(\beta)})$. Observe that
    $$
    \overline{\inte(\beta)} = \bigcup_{k=1}^n \overline{\inte(\gamma_k)}.
    $$
Write $x=P_{\gamma_1}x+(I-P_{\gamma_1})x$. Since $(I-P_{\gamma_1})x\in
\ker(P_{\gamma_1})=X_T(\overline{\mathrm{ext}(\gamma_1)})$, the hyperinvariance of the spectral subspaces yields
$$
(I-P_{\gamma_1})x\in X_T(\overline{\mathrm{int}(\beta)})\cap
X_T(\overline{\mathrm{ext}(\gamma_1)})
\subseteq X_T\!\left(\bigcup_{k=2}^n \overline{\mathrm{int}(\gamma_k)}\right).
$$
Hence $x=y_1+z_1$, where $y_1\in X_T(\overline{\mathrm{int}(\gamma_1)})$ and
$z_1\in X_T\!\left(\bigcup_{k=2}^n \overline{\mathrm{int}(\gamma_k)}\right)$.

\noindent Proceeding inductively, we obtain
$$
x=\sum_{k=1}^n y_k,
\qquad
y_k\in X_T(\overline{\mathrm{int}(\gamma_k)})=\operatorname{ran}(P_{\gamma_k}).
$$
Since $P_{\gamma_i}P_{\gamma_j}=0$ for $i\neq j$, it follows that
$$
P_\beta x
=\left(\sum_{k=1}^n P_{\gamma_k}\right)\sum_{k=1}^n y_k
=\sum_{k=1}^n y_k
=x.
$$
Therefore $x\in \operatorname{ran}(P_\beta)$, and hence
$$
\operatorname{ran}(P_\beta)=X_T(\overline{\mathrm{int}(\beta)}).
$$

Now, to prove that $\ker(P_\beta) = X_T\left(\overline{\exte(\beta)}\cup \left(\cup_{i,j=1, \ i\neq j}^n \gamma_{i,j} \right) \right)$, we argue as in the proof of Proposition \ref{proposicion union curvas} and deduce that
    \begin{equation*}
        \begin{split}
            \ker(P_\beta) &= \ker\left(\sum_{k=1}^n P_{\gamma_k}\right) = \bigcap\limits_{k=1}^n \ker(P_{\gamma_k}) = \bigcap\limits_{k=1}^n X_T(\overline{\exte(\gamma_k)})  = X_T\left(\bigcap\limits_{k=1}^n \overline{\exte(\gamma_k)} \right) \\& = X_T\left(\overline{\exte(\beta)}\cup \left(\cup_{i,j=1, \ i\neq j}^n \gamma_{i,j} \right) \right).
        \end{split}
    \end{equation*}
which is the desired conclusion.
\end{proof}

The following result shows that one can consider the spectral cut associated with the intersection of the interiors of two admissible cycles, whenever this intersection is enclosed by an admissible cycle:

\begin{proposition}\label{proposicion interseccion curvas}
Let $T$ be a bounded linear operator on $X$ with the SVEP, and let $\Gamma_1,\Gamma_2$ be admissible cycles such that $T$ admits plain spectral cuts along them with associated projections $P_{\Gamma_1}$ and $P_{\Gamma_2}$. Suppose there exists an admissible cycle $\Gamma$ satisfying
$$
\mathrm{int}(\Gamma)=\mathrm{int}(\Gamma_1)\cap \mathrm{int}(\Gamma_2)\neq\varnothing.
$$
If $\overline{\mathrm{int}(\Gamma)}\cap\sigma(T)\neq\varnothing$ and $\overline{\mathrm{ext}(\Gamma)}\cap\sigma(T)\neq\varnothing$, then $T$ admits a plain spectral cut along $\Gamma$, and the associated projection is $P_\Gamma=P_{\Gamma_1}P_{\Gamma_2}.$
\end{proposition}

\begin{proof}
Let $P_\Gamma=P_{\Gamma_1}P_{\Gamma_2}\in \biconm{T}.$ We will show that $\ran(P_\Gamma) = X_T(\overline{\inte(\Gamma)})$ and $\ker(P_\Gamma) = X_T(\overline{\exte(\Gamma)})$, which will show the statement.

First, recall that $\ran(P_\Gamma)\subseteq  X_T(\overline{\inte(\Gamma_1)})\cap X_T(\overline{\inte(\Gamma_2)}) = X_T(\overline{\inte(\Gamma)}).$ Now, if $x\in X_T(\overline{\inte(\Gamma)})=X_T(\overline{\inte(\Gamma_1)})\cap X_T(\overline{\inte(\Gamma_2)}),$  there exist $y,z\in X$ such that $x=P_{\Gamma_1}y=P_{\Gamma_2}z.$ We have
$$
P_\Gamma x = P_{\Gamma_1}P_{\Gamma_2}x = P_{\Gamma_1}P_{\Gamma_2}z = P_{\Gamma_1}y = x,
$$
so $\ran(P_\Gamma) =X_T(\overline{\inte(\Gamma)}).$

Now, let $x\in X_T(\overline{\exte(\Gamma)})$. Observe that $P_\Gamma x = P_{\Gamma_1}P_{\Gamma_2}x = P_{\Gamma_2}P_{\Gamma_1}x$ belongs to
$$X_T(\overline{\exte(\Gamma)})\cap X_T(\overline{\inte(\Gamma_1)})\cap X_T(\overline{\inte(\Gamma_2)}) = X_T(\Gamma) \subseteq  X_T(\Gamma_1) = \{0\}, $$ so $x\in \ker(P_\Gamma)$.

To prove that $\ker(P_\Gamma) \subseteq   X_T(\overline{\exte(\Gamma)})$,  note that $X_T(\overline{\exte(\Gamma_1)})+X_T(\overline{\exte(\Gamma_2)})\subseteq   X_T(\overline{\exte(\Gamma)})$. Now, let $x\in \ker(P_\Gamma)$ and write $x=P_{\Gamma_1}x+(I-P_{\Gamma_1})x$. Observe that $P_{\Gamma_2}P_{\Gamma_1} x =0$, so $P_{\Gamma_1}x \in \ker(P_{\Gamma_2}) = X_T(\overline{\exte(\Gamma_2)})$. Since $(I-P_{\Gamma_1})x \in X_T(\overline{\exte(\Gamma_1)})$, it follows that $x\in X_T(\overline{\exte(\Gamma_1)})+X_T(\overline{\exte(\Gamma_2)})\subseteq   X_T(\overline{\exte(\Gamma)})$ which concludes the proof.
\end{proof}

\medskip

We conclude the section by showing that a bounded linear operator $T$ with the SVEP that admits plain spectral cuts along a sufficiently rich family of horizontal and vertical lines, not necessarily forming an exhaustive cutting family, is decomposable.

\medskip

Assume $T$ is a linear bounded operator with the SVEP such that $\sigma(T)\subseteq  \D$ and $\sigma(T)$ does not lie in a horizontal or a vertical line. Define
    $$
    a=\min_{z\in \sigma(T)}\PR(z), \quad b=\max_{z\in \sigma(T)}\PR(z), \quad c=\min_{z\in \sigma(T)}\Impart(z), \quad d=\max_{z\in \sigma(T)}\Impart(z).
    $$
Suppose that for all $\x$ in a dense subset $A\subseteq  (a,b)$,  $T$ admits a plain spectral cut along the curve $\gamma_\x$, defined as
    $$     \gamma_\x :=\ell_\x \cup A_\x,     $$ where $$\ell_\x := \{z\in\overline{\D}: \PR(z)=\x\}, \qquad A_\x := \{z\in \T: \PR(z)\geq \x\}.$$
Likewise, suppose  that for all $\y$ in a dense subset $B\subseteq  (c,d)$, $T$ admits a plain spectral cut along the curve $\gamma_\y$, defined as
    $$     \gamma_\y :=\ell_\y \cup A_\y,     $$ where $$\ell_\y := \{z\in\overline{\D}: \Impart(z)=\y\}; \qquad A_\y := \{z\in \T: \Impart(z)\geq \y\}.$$
As a consequence of Proposition \ref{proposicion interseccion curvas}, it follows that for all $\x_1<\x_2$ in $A$ and $\y_1<\y_2$ in $B$, the operator $T$ admits a plain spectral cut along the curve $\gamma$ surrounding the rectangle $R=[\x_1,\x_2]\times [\y_1,\y_2]$, whenever both $R$ and $\mathbb{C}\setminus R$ intersect $\sigma(T)$. The associated projection is denoted by $P_R$.

\smallskip

Let us consider the family
    $$
    \mathcal{R} = \{[\x_1,\x_2]\times[\y_1,\y_2]\subset \C: \x_1,\x_2 \in A, \y_1,\y_2 \in B, \x_1<\x_2, \y_1<\y_2  \}.
    $$
Clearly, $T$ admits a plain spectral cut along the boundaries of all $R\in \mathcal{R}$ (whenever $R$ and $\C\setminus R$ intersect $\sigma(T)$) as well as the boundaries of their finite unions. Now, by Propositions \ref{proposicion union curvas} and \ref{proposicion union curvas que cortan}, if $R_1,\cdots,R_n \in \mathcal{R}$, there exists a projection $P_\beta = \sum_{k=1}^n P_{R_k}\in \biconm{T}$ with $\ran(P_\beta)=X_T(\bigcup_{k=1}^n R_k)$.

\smallskip

We now prove that $T$ is decomposable. Let $U,V$ be non-empty open subsets of $\mathbb{C}$ such that $\sigma(T)\subseteq U\cup V$. Choose open sets $G_1,G_2\subseteq \mathbb{C}$ with $\overline{G_1}\subseteq U$, $\overline{G_2}\subseteq V$, and $\sigma(T)\subseteq G_1\cup G_2$. Following Example~\ref{ejemplo 1}, consider a grid on $[a,b]\times[c,d]$ determined by partitions
$$
a=x_1<\cdots<x_n=b, \qquad c=y_1<\cdots<y_m=d,
$$
where $x_i\in A$ and $y_j\in B$ for all $i,j$, and whose mesh is smaller than $\delta/4$.

\smallskip

Define
    $$
    \sigma_1 = \bigcup_{([x_k,x_{k+1}]\times [y_\ell,y_{\ell+1}])\cap\overline{G_1} \neq \varnothing } [x_k,x_{k+1}]\times [y_\ell,y_{\ell+1}];
    $$
     $$
    \sigma_2 = \bigcup_{([x_k,x_{k+1}]\times[y_\ell,y_{\ell+1}])\cap\overline{G_2} \neq \varnothing } [x_k,x_{k+1}]\times [y_\ell,y_{\ell+1}].
    $$
Then, the projections given by
$$P_1 = \sum_{([x_k,x_{k+1}]\times [y_\ell,y_{\ell+1}])\cap\overline{G_1} \neq \varnothing } P_{\partial([x_k,x_{k+1}]\times [y_\ell,y_{\ell+1}])} \in \biconm{T}$$
$$P_2 = \sum_{([x_k,x_{k+1}]\times [y_\ell,y_{\ell+1}])\cap\overline{G_2} \neq \varnothing } P_{\partial([x_k,x_{k+1}]\times [y_\ell,y_{\ell+1}])} \in \biconm{T}$$
satisfy  $\ran(P_i)  = X_T(\sigma_i)$.

It is clear that $\sigma(T\mid_{ X_T(\sigma_1)}) \subseteq \sigma_1\cap \sigma(T)\subseteq \overline{U}$ and $\sigma(T\mid_{ X_T(\sigma_2)}) \subseteq \sigma_2\cap \sigma(T)\subseteq \overline{V}$.

Hence,  $T$ will be decomposable if we show that $X= X_T(\sigma_1)+X_T(\sigma_2).$ By Proposition \ref{proposicion union curvas que cortan}
    $$\sum_{([x_k,x_{k+1}]\times[y_\ell,y_{\ell+1}])\cap\sigma(T)\neq \varnothing } P_{\partial([x_k,x_{k+1}]\times [y_\ell,y_{\ell+1}])} = I_X,$$ since $$
    \sigma(T) \subseteq \bigcup_{([x_k,x_{k+1}]\times {[y_\ell,y_{\ell+1}]})\cap\sigma(T)\neq \varnothing } [x_k,x_{k+1}]\times {[y_\ell,y_{\ell+1}]}.
    $$
   At this point, observe that every rectangle of the grid that intersects $\sigma(T)$ cuts $\overline{G_1}$ or $\overline{G_2},$ so
   $$
   X= \ran\left(\sum_{([x_k,x_{k+1}]\times {[y_\ell,y_{\ell+1}]})\cap\sigma(T)\neq \varnothing } P_{[x_k,x_{k+1}]\times [y_\ell,y_{\ell+1}]} \right) \subseteq \ran(P_1)+\ran(P_2)=X_T(\sigma_1)+X_T(\sigma_2),
   $$
   and therefore $T$ is a decomposable operator, as claimed.

\section{Plain spectral cuts and local resolvent growth}\label{sec 4}

In this section we show that operators satisfying a suitable local resolvent growth condition along a curve $\gamma$ admit plain spectral cuts. In addition, we prove that the associated unconventional functional calculus admits an explicit integral representation. We further show that, for a broad class of operators that may fail to admit a plain spectral cut, the same local resolvent growth condition still yields a weaker form of spectral cut and, consequently, a corresponding unconventional functional calculus.
As an application of these ideas, we obtain a result that produces non-trivial closed hyperinvariant subspaces for the operator, extending previous results by Chalendar in \cite{Chalendar1997}.

\medskip

Before stating the main result of this section, recall that a bounded linear operator $T$ has the \emph{Dunford property $(C)$} if for every closed subset $F\subseteq \mathbb{C}$ the subspace $X_T(F)$ is closed in $X$.

\medskip

\begin{theorem}\label{teorema punto expuesto}
Let $X$ be a reflexive Banach space and let $T$ be a bounded linear operator on $X$ with the Dunford property $(C)$. Let $\gamma$ be a positively oriented rectifiable Jordan curve such that
\begin{enumerate}
    \item[(i)] $\gamma\cap \sigma(T)$ is non-empty and has zero arc-length measure;
    \item[(ii)] $\inte(\gamma)\cap\sigma(T)\neq\varnothing$ and $\exte(\gamma)\cap\sigma(T)\neq\varnothing$.
\end{enumerate}
Assume that, for each $x\in X$ and $y\in X^*$, the map
\begin{equation}\label{condicion integral}
z\in \gamma \longmapsto \langle (z I-T)^{-1}x,y\rangle
\end{equation}
is integrable. Then $T$ admits a plain spectral cut along $\gamma$ if and only if $X_T(\gamma)=\{0\}$. In this case, the associated idempotent $P_\gamma$ is given by
$$
P_\gamma x=\frac{1}{2\pi i}\int_\gamma (z I-T)^{-1}x\,dz,
\qquad x\in X.
$$
\end{theorem}

\smallskip

\noindent Before proving Theorem \ref{teorema punto expuesto}, a few comments are in order:

\begin{enumerate}
\item The reflexivity of $X$ is assumed in order to ensure that the Dunford and Pettis integrals on $X$ coincide. This guarantees that condition \eqref{condicion integral} implies that the vector
$$
\frac{1}{2\pi i}\int_\gamma (z I-T)^{-1}x\,dz
$$
is well defined and belongs to $X$ for every $x\in X$. If $X$ is not reflexive, condition \eqref{condicion integral} may be replaced by the stronger requirement
\begin{equation}\label{non-reflexive}
\int_\gamma \|(z I-T)^{-1}x\|\,|dz|<\infty,
\end{equation}
for all $x\in X$.

\item The hypothesis of $T$ enjoying Dunford's property $(C)$ in Theorem \ref{teorema punto expuesto} may be relaxed to merely requiring that $\Xint$ and $\Xext$ are closed.

\item Let $T$ be a bounded linear operator on $X$, and let $\gamma$ be a positively oriented rectifiable Jordan curve that intersects $\sigma(T)$ only at finitely many points $\{z_1,\ldots,z_n\}$. Assume moreover that both $\mathrm{int}(\gamma)\cap\sigma(T)$ and $\mathrm{ext}(\gamma)\cap\sigma(T)$ are non-empty. Then $\gamma$ satisfies hypotheses (i) and (ii) of Theorem~\ref{teorema punto expuesto}. In this situation, the condition $X_T(\gamma)=\{0\}$ is equivalent to requiring that $X_T(\{z_k\})=\{0\}$ for all $k=1,\ldots,n$.

Recall also that, with respect to the integrability condition \eqref{condicion integral}, the quantity $\|(z I-T)^{-1}x\|$ remains bounded when $z$ stays away from the points $z_1,\ldots,z_n$. Consequently, it suffices to require the integrability of $\|(z I-T)^{-1}x\|$ only on the portions of $\gamma$ lying in neighborhoods of $z_1,\ldots,z_n$.

\item Under the hypotheses of Theorem~\ref{teorema punto expuesto}, note that $(z I-T)^{-1}x$ coincides with the local resolvent function $x_T(z)$ for all $z\in\gamma\setminus\sigma(T)$. Since $\gamma\cap\sigma(T)$ is assumed to have zero arc-length measure, condition \eqref{condicion integral} is equivalent to requiring that the map $z\in\gamma\mapsto \pe{x_T(z),y}$ be integrable for all $x\in X$ and $y\in X^*$.
\end{enumerate}

\medskip

\medskip

\begin{demode}\emph{ Theorem \ref{teorema punto expuesto}.}
Let us prove that $X= X_T(\overline{\inte(\gamma)})\oplus X_T(\overline{\exte(\gamma)})$ as a topological direct sum if and only if $X_T(\gamma)=\{0\}$, and the statement of the theorem will follow.

First, observe that
$$
X_T(\overline{\operatorname{int}(\gamma)}) \cap X_T(\overline{\operatorname{ext}(\gamma)})
= X_T(\gamma \cap \sigma(T)).
$$
Therefore, if $X_T(\gamma \cap \sigma(T)) \neq \{0\}$, then $T$ does not admit a plain spectral cut along $\gamma$.

\smallskip

Now assume that $X_T(\gamma \cap \sigma(T)) = \{0\}$. Since $T$ is assumed to satisfy Dunford's property $(C)$ it suffices to show that
$$
X = X_T(\overline{\operatorname{int}(\gamma)}) + X_T(\overline{\operatorname{ext}(\gamma)}).
$$

\smallskip

Let $\tau$  be a rectifiable Jordan curve with positive orientation such that $\sigma(T)\subseteq \inte(\tau)$ and $\overline{\inte(\gamma)}\subseteq \inte(\tau)$ and consider the admissible cycle $\Gamma$ formed by $\tau$ and the curve $\gamma$ with negative orientation. Now, let us consider, for each vector $x\in X$, the vectors
$$
x^+ = \frac{1}{2\pi i} \int_\gamma (z I-T)^{-1}x\ dz, \qquad x^- = \frac{1}{2\pi i} \int_\Gamma (z I-T)^{-1}x\ dz.
$$
Recall that $(z I-T)^{-1}x$ is well defined at almost every point $z$ of both $\gamma$ and $\Gamma$, so the vectors $x^+$ and $x^-$ are well defined.
Since the map $z\in \gamma \mapsto (z I-T)^{-1}x$ is Dunford integrable, the reflexivity of $X$ and \cite[Theorem 11.55]{AB} yield that both $x^+$ and $x^-$ belong to $X$.  Now, let us show that $x^+ \in X_T(\overline{\inte(\gamma)})$. For all $w\in \rho(T)\setminus \overline{\inte(\gamma)}$,
$$(w I-T)^{-1}(z I-T)^{-1} = ( (z I-T)^{-1}-(w I-T)^{-1})(w-z)^{-1},$$
so
$$
(w I-T)^{-1}x^+ = \frac{1}{2\pi i}\int_{\gamma} \frac{1}{w-z}(z I-T)^{-1}x \ dz - \frac{1}{2\pi i}\int_{\gamma} \frac{1}{w-z}(w I-T)^{-1}x \ dz.
$$
Now, the second integral vanishes by Cauchy's Theorem, while the first integral is clearly analytic on $\exte(\gamma)$. This implies that $(w I-T)^{-1}x^+$ admits an analytic extension to $\C \setminus \overline{\inte(\gamma)}$, and hence $x^+ \in X_T(\overline{\inte(\gamma)})$.

Analogously, one can show that $x^- \in X_T(\overline{\exte(\gamma)})$. Therefore,
$$
X = X_T(\overline{\inte(\gamma)}) + X_T(\overline{\exte(\gamma)}),
$$
as desired.

\smallskip

To see that the spectral cut is non-trivial, suppose that
$X_T(\overline{\inte(\gamma)}) = \{0\}$.
Then $X = X_T(\overline{\exte(\gamma)})$, which is a contradiction, since it would imply that
$\sigma(T) \subseteq \overline{\exte(\gamma)}$.
However, this is impossible because $\sigma(T) \cap \inte(\gamma) \neq \varnothing$.
An analogous argument shows that $X_T(\overline{\exte(\gamma)}) \neq \{0\}$.

\smallskip

Finally, let $P_\gamma$ be the bounded idempotent such that
$$
\ran(P_\gamma) = X_T(\overline{\inte(\gamma)})
\quad\text{and}\quad
\ker(P_\gamma) = X_T(\overline{\exte(\gamma)}).
$$
Let $x\in X$ and write $x = x^+ + x^-$ as above. Recalling that
$x^+ \in X_T(\overline{\inte(\gamma)})$ and $x^- \in X_T(\overline{\exte(\gamma)})$, we obtain
$$
P_\gamma x = P_\gamma x^+ + P_\gamma x^- = x^+
= \frac{1}{2\pi i} \int_\gamma (zI - T)^{-1}x \, dz,
$$
which completes the proof.
\end{demode}

\smallskip

\smallskip

The following result provides an explicit integral representation for the unconventional functional calculus associated with the plain spectral cuts obtained in Theorem~\ref{teorema punto expuesto} extending, in particular, Theorem~\ref{teorema expresion integral calculo funcional inconvencional}.

\begin{corollary}\label{corolario calculo funcinonal local resolvent}
Let $X$ be a reflexive Banach space and let $T\in \mathcal{L}(X)$ satisfy Dunford's property $(C)$. Let $\gamma$ be a positively oriented rectifiable Jordan curve such that
\begin{enumerate}
\item[(i)] $\gamma \cap \sigma(T)$ is non-empty and has zero arc-length measure;
\item[(ii)] $\inte(\gamma)\cap\sigma(T)\neq\varnothing$ and $\exte(\gamma)\cap\sigma(T)\neq\varnothing$.
\end{enumerate}
Assume that for every $x\in X$ the map
$$
z\in\gamma \longmapsto \|(zI-T)^{-1}x\|
$$
is integrable and that $X_T(\gamma)=\{0\}$. Let $f:G\to\mathbb{C}$ be holomorphic on an open set $G\subset\mathbb{C}$ containing $\overline{\inte(\gamma)}$. Then
    $$
f_{\gamma}(T)x=\frac{1}{2\pi i}\int_{\gamma} f(z)(zI-T)^{-1}x\,dz
    $$
for every $x\in X$.
\end{corollary}

\smallskip

\begin{proof}
By Theorem~\ref{teorema punto expuesto}, $T$ admits a plain spectral cut along $\gamma$ and, by Theorem~\ref{teorema calculo bien definido}, the operator $f_\gamma(T)$ is well defined and
$$
f_\gamma(T)=f(T|_{\ran(P_\gamma)})P_\gamma,
$$
where $P_\gamma$ denotes the idempotent associated with this spectral cut.
Let $\tau$ be a rectifiable Jordan curve such that
$$
\overline{\inte(\gamma)}\subseteq \inte(\tau)\subseteq G,
$$
and let $\Gamma$ be the admissible cycle formed by $\tau$ and $\gamma$, with $\gamma$ taken with negative orientation.

\smallskip
Now recall that $f_\gamma(T)=f(T|_{\ran(P_\gamma)})P_\gamma$, and observe that
$\sigma(T|_{\ran(P_\gamma)})\subseteq \inte(\tau)$. Let $x\in\ran(P_\gamma)$. Then
$$
\begin{aligned}
f_\gamma(T)x
&= \frac{1}{2\pi i}\int_\tau f(z)(zI-T|_{\ran(P_\gamma)})^{-1}x\,dz  \\
&= \frac{1}{2\pi i}\int_\tau f(z)(zI-T)^{-1}x\,dz  \\
&= \frac{1}{2\pi i}\left(
\int_\gamma f(z)(zI-T)^{-1}x\,dz
+
\int_\Gamma f(z)(zI-T)^{-1}x\,dz
\right).
\end{aligned}
$$
Repeating the arguments from the proof of Theorem~\ref{teorema punto expuesto}, it follows that the two integrals in the last equality belong to $\ran(P_\gamma)$ and $\ker(P_\gamma)$, respectively. Since $f_\gamma(T)x\in\ran(P_\gamma)$, we conclude that
$$
\frac{1}{2\pi i}\int_\Gamma f(z)(zI-T)^{-1}x\,dz=0,
$$
and therefore
$$
f_\gamma(T)x=\frac{1}{2\pi i}\int_\gamma f(z)(zI-T)^{-1}x\,dz .
$$

Finally, using the expression for $P_\gamma$ given in Theorem~\ref{teorema punto expuesto} and applying Fubini's Theorem, we obtain, for every $x\in X$,
$$
f_\gamma(T)x
= \frac{1}{2\pi i}\int_\gamma f(z)(zI-T)^{-1}(P_\gamma x)\,dz
= \frac{1}{2\pi i}\int_\gamma (\xi I-T)^{-1}\!\left(
\frac{1}{2\pi i}\int_\gamma f(z)(zI-T)^{-1}x\,dz
\right)d\xi .
$$
Again, the inner integral belongs to $\ran(P_\gamma)$, and therefore
$$
f_\gamma(T)x
= \frac{1}{2\pi i}\int_\gamma f(z)(zI-T)^{-1}x\,dz,
$$
which completes the proof.
\end{proof}

\smallskip

We now consider the situation in which the integrability condition \eqref{condicion integral} holds only for vectors $x$ in a linear manifold $M\subseteq X$.

\begin{theorem}\label{teorema densidad punto expuesto}
Let $X$ be a reflexive Banach space and let $T\in \mathcal{L}(X)$ satisfy Dunford's property $(C)$. Let $\gamma$ be a positively oriented rectifiable Jordan curve such that
\begin{enumerate}
\item[(i)] $\gamma \cap \sigma(T)$ is non-empty and has zero arc-length measure;
\item[(ii)] $\inte(\gamma)\cap\sigma(T)\neq\varnothing$ and $\exte(\gamma)\cap\sigma(T)\neq\varnothing$.
\end{enumerate}
Let $M\subseteq X$ be a dense linear manifold and assume that, for each
$x\in M$ and $y\in X^*$, the map
$$
z\in \gamma \longmapsto \langle (zI-T)^{-1}x,\,y\rangle
$$
is integrable. Moreover, assume that there exists $C>0$ such that
\begin{equation}\label{acotacion norma integral}
\left\|\int_\gamma (zI-T)^{-1}x\,dz\right\|\le C\|x\|
\end{equation}
for all $x\in M$. Then $T$ admits a plain spectral cut along $\gamma$ if and only if
$X_T(\gamma)=\{0\}$. In this case, the associated idempotent $P_\gamma$
satisfies
$$
P_\gamma x=\frac{1}{2\pi i}\int_\gamma (zI-T)^{-1}x\,dz
$$
for all $x\in M$.
\end{theorem}

\begin{proof}
Arguing as in the proof of Theorem \ref{teorema punto expuesto}, it follows that if $X_T(\gamma)\neq\{0\}$, $T$ does not admit a plain spectral cut along $\gamma$.

For the converse, assume $X_T(\gamma) = \{0\}$ and let us show that $X=\Xint\oplus\Xext$ holds as a topological sum.

Note that the integrability of the map
$$
z\in \gamma \longmapsto \pe{(z I-T)^{-1}x,y}
$$
implies, by \cite[Theorem 11.55]{AB}, that
$$
\frac{1}{2\pi i}\int_\gamma (zI-T)^{-1}x\,dz
$$
is well defined and belongs to $X$ for every $x\in M$. Since $M$ is dense and
\eqref{acotacion norma integral} holds, the map
$$
x\in M \longmapsto \frac{1}{2\pi i}\int_\gamma (zI-T)^{-1}x\,dz
$$
extends to a bounded operator $J_\gamma$ on $X$.

The arguments in the proof of Theorem \ref{teorema punto expuesto} also yield that
$J_\gamma x\in \Xint$ for every $x\in M$ and
$$
(I-J_\gamma)x=\frac{1}{2\pi i}\int_\Gamma (zI-T)^{-1}x\,dz \in \Xext,
$$
where $\Gamma$ denotes the admissible cycle used in that proof.

Since $\Xint$ and $\Xext$ are closed, we have
$\ran(J_\gamma)\subseteq \Xint$ and $\ker(J_\gamma)\subseteq \Xext$.
Moreover, for every $x\in X$,
$$
x = J_\gamma x + (I-J_\gamma)x,
$$
and hence $X=\Xint+\Xext$. Since $\Xint\cap\Xext = X_T(\gamma)=\{0\}$, $ X=\Xint\oplus \Xext $ as a topological direct sum.

It remains to show that $J_\gamma$ coincides with the associated idempotent $P_\gamma$.
This follows directly from the uniqueness of the decomposition
$x=P_\gamma x+(I-P_\gamma)x$, and the proof is complete.
\end{proof}

An equivalent result to Corollary \ref{corolario calculo funcinonal local resolvent} holds for this setting:

\begin{corollary}
Let $X$ be a reflexive Banach space and let $T\in\mathcal{L}(X)$ have the Dunford property $(C)$. Let $\gamma$ be a positively oriented rectifiable Jordan curve such that
\begin{enumerate}
\item[(i)] $\gamma\cap\sigma(T)$ is non-empty and has zero arc-length measure;
\item[(ii)] $\inte(\gamma)\cap\sigma(T)$ and $\exte(\gamma)\cap\sigma(T)$ are both non-empty.
\end{enumerate}
Let $M\subset X$ be a dense linear manifold and assume that, for every $x\in M$ and $y\in X^*$, the map
$$
z\mapsto \langle (zI-T)^{-1}x,y\rangle ,\qquad z\in\gamma,
$$
is integrable. Suppose, moreover, that there exists $C>0$ such that
\begin{equation}\label{acotacion norma integral-}
\left\|\int_\gamma (zI-T)^{-1}x\,dz\right\|\le C\|x\|
\end{equation}
for all $x\in M$, and that $X_T(\gamma)=\{0\}$.  Let $f:G\to\C$ be holomorphic on an open set $G\subset\C$ containing $\overline{\inte(\gamma)}$. Then
$$
f_\gamma(T)x=\frac{1}{2\pi i}\int_\gamma f(z)(zI-T)^{-1}x\,dz
$$
for every $x\in M$.
\end{corollary}

The proof follows the same lines as that of Corollary \ref{corolario calculo funcinonal local resolvent}, with the obvious modifications, and is therefore omitted.

\medskip

We now provide an example of an operator satisfying the hypotheses of Theorem \ref{teorema densidad punto expuesto}, following some of the ideas by Klaja  \cite{Klaja}.

\begin{example}\label{ejemplo normal}
Let $K$ be the union of the two tangent closed discs $\overline{D}(-1,1)$ and
$\overline{D}(1,1)$, and consider the multiplication operator $T=M_z$ acting on
$X=L^p(K,dm)$, where $m$ denotes the Lebesgue measure on $K$ and
$1\le p<\infty$. It is well known that $T$ is a decomposable operator and
therefore has the Dunford property.

Let $\gamma$ be a rectifiable Jordan curve containing the segment $[-i,i]$ and
such that $\gamma\cap K=\{0\}$. Let $\mathcal{L}$ denote the linear manifold
consisting of all finite linear combinations of indicator functions $\chi_B$,
where $B$ is a Borel subset of $K$ at a positive distance from $\gamma$. It is
straightforward to verify that $\mathcal{L}$ is dense in $X$.

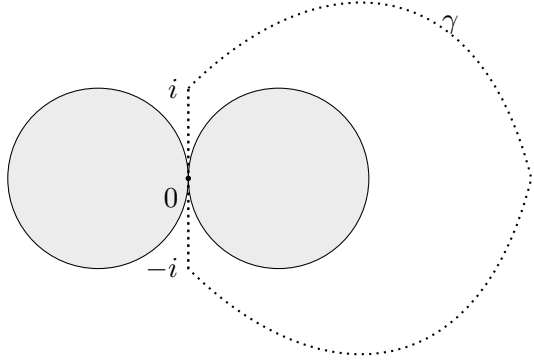
\begin{figure}[ht]
\hspace*{1,3cm}
\begin{tikzpicture}[scale=1.2]

\fill[gray!15] (-1,0) circle (1);
\fill[gray!15] (1,0) circle (1);
\draw (-1,0) circle (1);
\draw (1,0) circle (1);

\fill (0,0) circle (0.03);
\node[below left] at (0,0) {$0$};

\draw[dotted,thick] (0,-1) -- (0,1);
\node[left] at (0,1) {$i$};
\node[left] at (0,-1) {$-i$};

\draw[dotted,thick]
(0,1)
.. controls (1.4,2.4) and (3.2,2.4) .. (3.8,0)
.. controls (3.2,-2.4) and (1.4,-2.4) .. (0,-1)
-- cycle;

\node[right] at (2.7,1.7) {$\gamma$};

\end{tikzpicture}
\caption{The set $K=\overline{D}(-1,1)\cup\overline{D}(1,1)$ and a Jordan curve $\gamma$ containing $[-i,i]$ with $\gamma\cap K=\{0\}$.}
\end{figure}

Let us show that the map
$$
z\in \gamma \longmapsto \pe{(zI-T)^{-1}f,g}
$$
is integrable for every $f\in \mathcal{L}$ and $g\in X^*$. Let
$f=\sum_{k=1}^n a_k \chi_{B_k}\in \mathcal{L}$. It suffices to prove that
$\pe{(zI-T)^{-1}f,g}$ is integrable along the segment $[-i,i]$. Indeed,
\begin{equation*}
\begin{split}
\int_{-1}^1 \left|\pe{(it-T)^{-1}f,g}\right|\,dt
&\le \int_{-1}^1 \left( \sum_{k=1}^n |a_k|
\int_{B_k} \frac{|g(\xi)|}{|it-\xi|}\,dm(\xi) \right) dt .
\end{split}
\end{equation*}
Since the distance from each $B_k$ to $\gamma$ is positive, there exists a constant
$C>0$ such that
$$
\int_{-1}^1 \left( \sum_{k=1}^n |a_k|
\int_{B_k} \frac{|g(\xi)|}{|it-\xi|}\,dm(\xi) \right) dt
\le C \sum_{k=1}^n |a_k| \int_{B_k} |g(\xi)|\,dm(\xi)
\le C \|g\|_{L^{p'}(K,dm)} \sum_{k=1}^n |a_k| < \infty,
$$
where $p'$ denotes the conjugate exponent of $p$. Consequently, it follows that the function
$$
\xi \in K \longmapsto \int_\gamma (zI-T)^{-1}f(\xi)\,dz
= \int_\gamma \left(\sum_{k=1}^n a_k \frac{\chi_{B_k}(\xi)}{z-\xi}\right) dz
$$
belongs to $L^p(K,dm)$. Moreover, for almost every $\xi\in K$ we have
\begin{equation*}
\begin{split}
\int_\gamma \left(\sum_{k=1}^n a_k \frac{\chi_{B_k}}{z-\xi}\right)dz
&= \sum_{k=1}^n a_k \chi_{B_k}(\xi)\int_\gamma \frac{dz}{z-\xi} \\
&= 2\pi i\,\chi_{\inte(\gamma)}(\xi)\sum_{k=1}^n a_k \chi_{B_k}(\xi) \\
&= 2\pi i\,\chi_{\inte(\gamma)}(\xi) f(\xi).
\end{split}
\end{equation*}

\noindent Consequently,
$$
\left\|\int_\gamma (zI-T)^{-1}f\,dz\right\|
\le 2\pi\|f\|
$$
for every $f\in \EL$. Hence $T$ satisfies the hypotheses of Theorem \ref{teorema densidad punto expuesto}. Since
$$
X_T(\gamma)=X_T(\{0\})=\{0\},
$$
it follows that $T$ admits a plain spectral cut along $\gamma$. Moreover, the associated idempotent $P_\gamma$ is given by
$$
P_\gamma f=\frac{1}{2\pi i}\int_\gamma (zI-T)^{-1}f\,dz
$$
for all $f\in \EL$.
\end{example}

\medskip

\begin{remark} Recall that the operator $T$ from Example \ref{ejemplo normal} is a \textit{scalar} operator, i.e., there exists a spectral measure
$E$ defined on the Borel subsets of $\C$ such that
\begin{equation}\label{scalar operator}
T=\int_\C \lambda\, dE(\lambda).
\end{equation}
In this case (see \cite{LN00}), one has $E(\overline{\inte(\gamma)})=\Xint$.
Since this is a projection commuting with $T$, it follows that
$E(\overline{\inte(\gamma)})=P_\gamma$. Hence the plain spectral cuts yield
the same idempotents as those arising from the scalar spectral measure.

\noindent Thus our approach provides an explicit construction of the idempotent $P_\gamma$
without requiring \emph{a priori} a representation of the form \eqref{scalar operator}.
\end{remark}

\medskip

We close the section by discussing weaker forms of spectral cuts whenever $M$ is non-dense or the boundedness condition \eqref{acotacion norma integral-} does not hold.

\smallskip

\begin{theorem}\label{teorema linear manifold}
Let $X$ be a reflexive Banach space and let $T\in\mathcal \EL(X)$ have the SVEP.
Let $\gamma$ be a positively oriented rectifiable Jordan curve such that
\begin{enumerate}
\item[(i)] $\gamma\cap\sigma(T)\neq\varnothing$ and has zero arc-length measure;
\item[(ii)] $\inte(\gamma)\cap\sigma(T)\neq\varnothing$ and $\exte(\gamma)\cap\sigma(T)\neq\varnothing$.
\end{enumerate}
Let $\{0\}\subsetneq M\subseteq X$ be a linear manifold and assume that, for every
$x\in M$ and $y\in X^*$, the map
$$
z\in\gamma \mapsto \langle (zI-T)^{-1}x,y\rangle
$$
is integrable. Then for every $x\in M$ there exist $x^+\in\Xint$ and
$x^-\in\Xext$ such that $x=x^++x^-$. This decomposition is unique if and only if
$X_T(\gamma)=\{0\}$. In that case,
$$
x^+=\frac{1}{2\pi i}\int_\gamma (zI-T)^{-1}x\,dz,
\qquad
x^-=\frac{1}{2\pi i}\int_\Gamma (zI-T)^{-1}x\,dz,
$$
where $\Gamma$ denotes the admissible cycle formed by any positively oriented
Jordan curve $\tau$ surrounding $\sigma(T)$ together with $\gamma$ taken with
the clockwise orientation. Moreover, if $\Xint$ and $\Xext$ are closed, then:
\begin{enumerate}
\item[(a)] There exists a norm $\|\cdot\|_e$ on $\Xint\dotplus\Xext$ making it a Banach space.
\item[(b)] There exists an idempotent $J\in\mathcal{L}(\Xint\dotplus\Xext, \norm{\cdot}_e)$ with
$\ran(J)=\Xint$ and $\ker(J)=\Xext$.
\item[(c)] For every $x\in M$, $Jx=x^+$ and $(I-J)x=x^-$.
\end{enumerate}
\end{theorem}

\begin{proof}
Fix $x\in M$ and set
$$
x^+ = \frac{1}{2\pi i}\int_\gamma (zI-T)^{-1}x\,dz,
\qquad
x^- = \frac{1}{2\pi i}\int_\Gamma (zI-T)^{-1}x\,dz.
$$
As in the proof of Theorem \ref{teorema punto expuesto}, $x^+\in\Xint$, $x^-\in\Xext$, and
$$
x^++x^-=\frac{1}{2\pi i}\int_\tau (zI-T)^{-1}x\,dz=x.
$$
Since $X_T(\gamma)=\Xint\cap\Xext$, the decomposition $x=x^++x^-$ is unique if and only if $X_T(\gamma)=\{0\}$.

\smallskip

Now assume that $\Xint$ and $\Xext$ are closed. Since $X_T(\gamma)=\{0\}$, we have
$\Xint\dotplus\Xext$ as an algebraic direct sum. For $x\in \Xint\dotplus\Xext$ write
$x=y+z$ with $y\in\Xint$ and $z\in\Xext$, and define
$$
\|x\|_e=\|y\|_X+\|z\|_X .
$$
It is straightforward to verify that $\|\cdot\|_e$ is a norm on $\Xint\dotplus\Xext$
under which this space becomes a Banach space. Moreover, $\Xint$ and $\Xext$
are closed with respect to $\|\cdot\|_e$. Hence, $\Xint\oplus\Xext$ is a topological
direct sum with respect to $\norm{\cdot}_e$. Consequently, there exists an idempotent operator
$J:\Xint\oplus\Xext\to\Xint\oplus \Xext$, bounded with respect to $\|\cdot\|_e$, such that
$Jx\in\Xint$ and $(I-J)x\in\Xext$ for all $x$.

Finally, if $x\in M$, the uniqueness of the decomposition $x=x^++x^-$ together
with the fact that $J$ is idempotent implies that $Jx=x^+$ and $(I-J)x=x^-$.
This completes the proof.
\end{proof}

The following example illustrates situations in which Theorem \ref{teorema linear manifold} applies.

\begin{example}
    Let $T$ be a linear bounded operator on $X$ with the SVEP and assume that there exists a positively-oriented rectifiable Jordan curve $\gamma$ with $\gamma\cap\sigma(T) = \{z_0\}$. Assume that there exists a holomorphic function $\phi : U\rightarrow \C$, where $U$ is an open set containing $\sigma(T)$, such that the map
    $$z\in \gamma\setminus\{z_0\}\mapsto  \norm{\phi(z)(z I-T)^{-1}}$$
    is bounded. Now, let $x \in X$ and consider $h=\phi(T)x$. It follows that
    $$(z I-T)^{-1}h = (z I-T)^{-1}\phi(T)x = \frac{1}{2\pi i}\int_\gamma \phi(z)(z I-T)^{-1} x\ dz $$ for all $z\in \gamma\setminus\{z_0\}$. Thus
    $$
    \norm{(z I-T)^{-1}h} \leq \frac{\norm{x}}{2\pi} \int_\gamma \norm{\phi(z)(z I-T)^{-1}}\ dz < \infty.
    $$
  Accordingly, setting $M=\ran(\phi(T))$, for every $h\in M$ and $y\in X^*$ the map
$$
z\mapsto \langle (zI-T)^{-1}h,y\rangle, \qquad z\in\gamma,
$$
is integrable.
\end{example}
The ideas behind the previous construction go back to \cite{Stampfli}, where a boundedness condition on the map
$$
z\in \gamma\setminus\{z_0\}\longmapsto \|\phi(z)(zI-T)^{-1}\|
$$
is used to obtain non-trivial closed hyperinvariant subspaces for $T$.
Theorem \ref{teorema linear manifold} shows that related conditions on the growth of the resolvent of $T$ also yield the decomposition properties described above.
Thus, the following result may be regarded as a refinement of the conditions in \cite{Stampfli} and extends results in \cite{Chalendar1997, Chalendar1998}.

\begin{theorem}
Let $X$ be a reflexive Banach space and let $T\in\mathcal L(X)$ have the SVEP. Let $\gamma_1,\gamma_2$ be positively oriented rectifiable Jordan curves such that
\begin{enumerate}
\item[(i)] $\gamma_i\cap\sigma(T)\neq\varnothing$ and has zero arc-length measure for $i=1,2$;
\item[(ii)] $\inte(\gamma_i)\cap\sigma(T)\neq\varnothing$ and $\exte(\gamma_i)\cap\sigma(T)\neq\varnothing$ for $i=1,2$;
\item[(iii)] $\overline{\inte(\gamma_1)}\cap\overline{\inte(\gamma_2)}=\varnothing$.
\end{enumerate}
Assume further that:
\begin{enumerate}
\item[(a)] There exists $x\in X$ such that the map
\begin{equation}\label{int1}
z\mapsto \langle (zI-T)^{-1}x,y\rangle,\qquad z\in\gamma_1,
\end{equation}
is integrable for every $y\in X^*$, and
$$
\int_{\gamma_1}\langle (zI-T)^{-1}x,y_0\rangle\,dz\neq0
$$
for some $y_0\in X^*$;

\item[(b)] There exists $x^*\in X^*$ such that the map
\begin{equation}\label{int2}
z\mapsto \langle (zI-T^*)^{-1}x^*,h\rangle,\qquad z\in\gamma_2,
\end{equation}
is integrable for every $h\in X$, and
$$
\int_{\gamma_2}\langle (zI-T^*)^{-1}x^*,h_0\rangle\,dz\neq0
$$
for some $h_0\in X$.
\end{enumerate}
Then $T$ admits a non-trivial closed hyperinvariant subspace.
\end{theorem}

\begin{proof}
Arguing as in the proof of Theorem \ref{teorema punto expuesto}, it follows that the vectors
$$x_1 = \frac{1}{2\pi i}\int_{\gamma_1} (z I-T)^{-1}x dz, \qquad x^*_2 = \frac{1}{2\pi i} \int_{\gamma_2} (z I-T^*)^{-1}x^*dz $$
are well defined and belong to $X_T(\overline{\inte(\gamma_1)})$ and $X^*_{T^*}(\overline{\inte(\gamma_2)}),$ respectively. Moreover, the existence of $y_0$ and $h_0$ yields that such vectors are non-zero, so the previous spectral subspaces are both non-zero. Now, recall that $X_T(\overline{\inte(\gamma_1)})$ is hyperinvariant under $T$, so it suffices to show that it is non-dense to deduce that its closure is a non-trivial closed hyperinvariant subspace for $T$. By \cite[Proposition 2.5.1]{LN00}
    $$X_T(\overline{\inte(\gamma_1)}) \subset X^*_{T^*}(\overline{\inte(\gamma_2)})^\perp,
    $$
    and since $X_{T^*}(\overline{\inte(\gamma_2)})$ is non-zero, the statement follows.
\end{proof}

\begin{remark} If $X$ is not reflexive, the integrability conditions required in \eqref{int1} and \eqref{int2} have to be replaced by the requirements
$$
\int_{\gamma_1} \|(z I-T)^{-1}x\|\,|dz|<\infty \text{ and } \int_{\gamma_2} \|(z I-T^*)^{-1}x^*\|\,|dz|<\infty,
$$
respectively.
\end{remark}

\section{Plain spectral cuts for compact perturbations of normal operators}\label{sec 5}

Let $H$ be a complex infinite-dimensional separable Hilbert space. In this section we prove that a large class of trace-class perturbations of diagonalizable normal operators on $H$ admits an exhaustive cutting family. As a consequence, these operators admit an unconventional functional calculus and are super-decomposable, which extends previous results by the authors in \cite{GG3}.

In order to illustrate some of the core ideas, we first consider the case of cyclic normal operators.

\medskip

\subsection{A first example: plain spectral cuts for normal operators}
Let $\mu$ be a compactly supported positive finite Borel measure and consider the multiplication operator $T=M_z$ acting on $L^2(\mu)$.
It is well known that every cyclic normal operator on $H$ is unitarily equivalent to an operator of this form.

\smallskip

Let $\gamma$ be a rectifiable Jordan curve such that both $\inte(\gamma)$ and $\exte(\gamma)$ intersect $\sigma(M_z) = \supp(\mu)$ and $\mu(\gamma) = 0$. Then, the multiplication operator $M_{\chi_{\overline{\inte(\gamma)}}}$ is an idempotent in $\biconm{M_z}$ such that
\begin{enumerate}
    \item [(i)] $\ran(M_{\chi_{\overline{\inte(\gamma)}}}) = \chi_{\overline{\inte(\gamma)}}L^2(\mu) = X_T(\overline{\inte(\gamma)}).$
    \item [(ii)] $I-M_{\chi_{\overline{\inte(\gamma)}}} = M_{\chi_{\overline{\exte(\gamma)}}}$ and $\ker(M_{\chi_{\overline{\inte(\gamma)}}}) = X_T(\overline{\exte(\gamma)}).$
\end{enumerate}
Thus, $M_z$ admits a plain spectral cut along $\gamma$, with associated idempotent $P_\gamma = M_{\chi_{\overline{\inte(\gamma)}}}.$

\smallskip

We show that the unconventional functional calculus associated with $\gamma$ admits an explicit integral representation on a dense subspace $\EL\subset L^2(\mu)$.
Let $\EL$ be the linear manifold of finite linear combinations of indicator functions $\chi_B$, where $B\subset\supp(\mu)$ is a Borel set with positive distance from $\gamma$.
As in Example \ref{ejemplo normal}, $\EL$ is dense in $L^2(\mu)$.

\smallskip

Take $f= \sum_{k=1}^n a_k \chi_{B_k}\in \EL$, and observe that for all $\zeta\in \gamma$, the function
$$(\zeta I-M_z)^{-1} f = \sum_{k=1}^n a_k \frac{\chi_{B_k}}{\zeta-z}$$
is well defined and lies in $L^2(\mu)$. Moreover, it follows that the map
$$\zeta \in \gamma \mapsto \pe{(\zeta I-M_z)^{-1} f,g}$$
is integrable for every  $f\in\EL$ and $g\in L^2(\mu)$. In addition,
$$
\frac{1}{2\pi i}\int_\gamma (\zeta I-M_z)^{-1} f(z) d\zeta = \chi_{\overline{\inte(\gamma)}}(z)f(z)
$$
for almost every $z\in \supp(\mu)$, so, there exists a constant $C>0$ such that
$$
\norm{\int_\gamma (\zeta I-M_z)^{-1}f d\zeta}_{L^2(\mu)} \leq C \norm{f}_{L^2(\mu)}
$$
for all $f\in\EL$. Thus, the map
$$f\in \EL \longmapsto \frac{1}{2\pi i}\int_\gamma (\zeta I-M_z)^{-1}f d\zeta$$
extends to a linear bounded operator acting on $L^2(\mu)$ that coincides with $P_\gamma$.

\medskip

Finally, let $G\supset\overline{\inte(\gamma)}$ be open and let $F:G\to\C$ be holomorphic.
For $f=\sum_{k=1}^n a_k\chi_{B_k}\in\EL$ and almost every $z\in\supp(\mu)$,
$$
\frac{1}{2\pi i}\int_\gamma F(\zeta)(\zeta I-M_z)^{-1}f(z)\,d\zeta
= \chi_{\overline{\inte(\gamma)}}(z)F(z)f(z)
= F(M_z|_{\ran(P_\gamma)})P_\gamma f(z)
= F_\gamma(M_z)f(z).
$$
Thus the unconventional functional calculus admits the claimed integral representation.

\medskip

\begin{remark}
The unconventional functional calculus for $M_z$ may be viewed as intermediate between the Dunford and the Borel functional calculi: it is richer than the former, but admits fewer functions than the latter.
\end{remark}

\subsection{Trace class perturbations of diagonalizable normal operators}
In this subsection, we deal with  operators that are unitarily equivalent to
$$T=D_\Lambda + K\in \EL(H),$$
where $D_\Lambda \in \EL(H)$ is a diagonal operator with respect to an orthonormal basis $(e_n)_{n\in \N}\subset H$ of eigenvectors, $\Lambda = (\lambda_n)_{n\in \N}$ is the set of associated eigenvalues, and
$$ K =\sumk u_k\otimes v_k,$$ where $u_k = \sumn \al_n^{(k)}e_n$ and $v_k = \sumn \beta_n^{(k)}$ are non-zero vectors in $H$ satisfying
$$\sumk \norm{u_k}\norm{v_k} < \infty.$$
We consider the subclass of these operators consisting of $T=D_\Lambda + \sumk u_k\otimes v_k$ such that
\begin{equation}\label{summability condition}
			 \sum_{(n,k) \in \mathcal{N}_u} |\al_n^{(k)}|^2\log \left(1+\frac{1}{|\al_n^{(k)}|}\right) + \sum_{(n,k) \in \mathcal{N}_v}|\beta_n^{(k)}|^2\log\left(1+ \frac{1}{|\beta_n^{(k)}|}\right) < \infty,
			 \end{equation}
			 where $\mathcal{N}_u := \{(n,k) \in \N\times \N : \al_n^{(k)} \neq 0 \}$ and $\mathcal{N}_v := \{(n,k) \in \N\times \N : \beta_n^{(k)} \neq 0 \}$. Moreover, we will also assume that $\sigma_p(T)\cup \sigma_p(T^*)$ is at most countable and $\Lambda'$ is not reduced to a singleton. Observe that  condition \eqref{summability condition} implies that
\begin{equation}\label{consecuencia sumabilidad 1}
\sumn \sumk \left(|\al_n^{(k)}|^2 + |\beta_n^{(k)}|^2 \right) < \infty
\end{equation}
and
	\begin{equation}\label{consecuencia sumabilidad 2}
		\sum_{(n,k) \in \mathcal{N}_u} |\al_n^{(k)}|^2\log \left(\frac{1}{|\al_n^{(k)}|}\right) + \sum_{(n,k) \in \mathcal{N}_v}|\beta_n^{(k)}|^2\log\left( \frac{1}{|\beta_n^{(k)}|}\right) < \infty.
	\end{equation}
In particular, \eqref{consecuencia sumabilidad 1} yields that $\sumk (\norm{u_k}^2 + \norm{v_k}^2) < \infty,$ so $K$ is  trace-class.

\medskip

In what follows, our main objective is to construct an exhaustive cutting family for the operators described above. To this end, we introduce the \emph{decomposability set} of $T$, which is a slight modification of the notion presented in \cite{GG3}.

	\begin{definition}\label{definition decomposability set}
		Let $T=D_\Lambda + \sumk u_k\otimes v_k \in \EL(H)$. Assume that $\sigma_p(T)\cup \sigma_p(T^*)$ is at most countable, $\Lambda'$ is not reduced to a singleton and that $\Lambda$ does not lie in a vertical or a horizontal line. Then, the \textit{decomposablity set} $\Delta(T)$ of $T$ consists of all real numbers $\x \in \R \setminus ( \PR(\sigma_p(T)\cup\sigma_p(T^*))\cup \Impart(\sigma_p(T)\cup\sigma_p(T^*)))$ such that
		
		$$ \sumn \sumk \frac{|\al_n^{(k)}|^2}{|\PR(\lambda_n)-\x|} +\frac{|\beta_n^{(k)}|^2}{|\PR(\lambda_n)-\x|} + \frac{|\al_n^{(k)}|^2}{|\Impart(\lambda_n)-\x|} +\frac{|\beta_n^{(k)}|^2}{|\Impart(\lambda_n)-\x|} < \infty. $$
	\end{definition}

Let us clarify several aspects of this definition. First, observe that $\Delta(T)$ contains no eigenvalues of $T$ or $T^*$. Moreover, if $T$ satisfies \eqref{summability condition}, then \cite[Lemma 2.3]{GG3} (whose proof works identically for the case of the imaginary part) ensures that $\Delta(T)$ contains almost every point of $\mathbb{R}$. In particular, $\Delta(T)\cap \PR(\Lambda')$ and $\Delta(T)\cap \Impart(\Lambda')$ contain almost every point of $\PR(\Lambda')$ and $\Impart(\Lambda')$, respectively.

\smallskip

We also assume that $\Lambda$ is not contained in any vertical or horizontal line in order to guarantee that $\Delta(T)$ is well defined. This is a harmless assumption, since it can always be achieved by multiplying $T$ by an appropriate $e^{i\theta}\in \mathbb{T}$.	
	
\medskip

Now, suppose $T=D_\Lambda + \sumk u_k\otimes v_k \in \EL(H)$ satisfies that $\sigma_p(T)\cup \sigma_p(T^*)$ is at most countable, $\Lambda'$ is not reduced to a singleton and that $\Lambda$ does not lie in a vertical or a horizontal line and set
$$
a= \min \PR(\sigma(T))-1, \quad b= \max \PR(\sigma(T))+1, \quad
c= \min \Impart(\sigma(T))-1, \quad d= \max \Impart(\sigma(T))+1.
$$
Clearly,  $\sigma(T)$ is contained in the rectangle $(a,b)\times (c,d)$. We will say that $G:=G_1\times G_2$ is an \textit{appropriate grid}  for $T$ if there exists a finite collection of points
$$
a= x_1 < x_2 < \cdots < x_n = b, \qquad
c= y_1 < y_2 < \cdots < y_m = d
$$
such that $G_1=\{x_i\}_{i=1}^n$, $G_2=\{y_j\}_{j=1}^m$ and $(x_i,y_j) \in \Delta(T)\times\Delta(T)$ for all
$i=1,\ldots,n$ and $j=1,\ldots,m$.

\begin{figure}[h!]
    \centering
 \hspace*{1cm}   \begin{tikzpicture}[font=\large]

    \def\valA{1} 
    \def\valB{6} 
    \def\valC{1} 
    \def\valD{5} 
    \def\stepGrid{0.5} 

    \draw[->, thick] (-0.5,0) -- (\valB+1,0) node[right] {$\mathrm{Re}(z)$};
    \draw[->, thick] (0,-0.5) -- (0,\valD+1) node[above] {$\mathrm{Im}(z)$};

    \draw[dotted] (\valA,\valC) -- (\valA,0) node[below] {$a$};
    \draw[dotted] (\valB,\valC) -- (\valB,0) node[below] {$b$};
    \draw[dotted] (\valA,\valC) -- (0,\valC) node[left] {$c$};
    \draw[dotted] (\valA,\valD) -- (0,\valD) node[left] {$d$};

    \begin{scope}
        \clip (\valA,\valC) rectangle (\valB,\valD);

        \draw[step=\stepGrid, black!30, thin] (\valA,\valC) grid (\valB,\valD);

        \pgfmathsetmacro{\secondX}{\valA+\stepGrid}
        \pgfmathsetmacro{\secondY}{\valC+\stepGrid}

        \foreach \x in {\valA, \secondX, ..., \valB} {
            \foreach \y in {\valC, \secondY, ..., \valD} {
                \fill[black] (\x,\y) circle (1.5pt);
            }
        }
    \end{scope}

    \draw[thick, black] (\valA,\valC) rectangle (\valB,\valD);
    \node[black, anchor=south east] at (\valB, \valD) {$[a,b] \times [c,d]$};

    \draw[fill=black!20, draw=black, thick, fill opacity=0.8, even odd rule]
        plot [smooth cycle, tension=0.7] coordinates {
            (2, 2) (2.5, 4) (4.5, 3.5) (5, 2.5) (3.5, 1.5)
        }
        (3, 3) circle (0.3cm)
        (4.2, 2.3) ellipse (0.4cm and 0.2cm);

    \node
        at (4.1, 3.3) {$\sigma(T)$};

\end{tikzpicture}
\hspace{-1cm}\caption{An example of an admissible grid for $T$} \label{grid}
\end{figure}
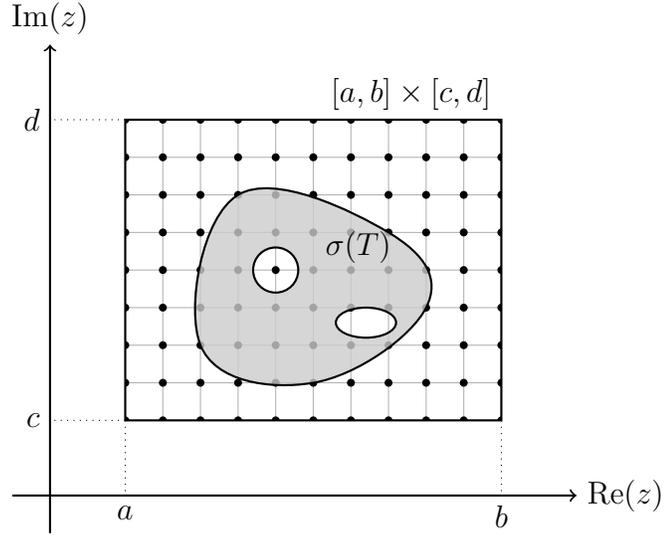

\begin{definition}
Let $T$ be an operator as in Definition \ref{definition decomposability set}. A Jordan curve $\gamma$ is said to be \textit{appropriate} for $T$ if:
\begin{enumerate}
    \item[(i)] $\gamma$ is a polygonal curve whose segments are parallel to either the real or the imaginary axis.
    \item[(ii)] There exists an admissible grid $G$ for $T$ such that the endpoints of the segments forming $\gamma$ lie on $G$.
    \item[(iii)] Both $\inte(\gamma)$ and $\exte(\gamma)$ intersect $\sigma(T)$.
\end{enumerate}
A cycle $\Gamma$ is \textit{appropriate} for $T$ if it is an admissible cycle for $T$ formed by a finite union of curves, each of which is either an appropriate curve for $T$ or $\partial([a,b]\times[c,d])$. We denote by $\mathcal{F}_T$ the family of all appropriate cycles for $T$.\end{definition}
\begin{figure}[h!]
    \centering
 \hspace*{1cm}   \begin{tikzpicture}[font=\large]

    \def\valA{1} 
    \def\valB{6} 
    \def\valC{1} 
    \def\valD{5} 
    \def\stepGrid{0.5} 

    \draw[->, thick] (-0.5,0) -- (\valB+1,0) node[right] {$\mathrm{Re}(z)$};
    \draw[->, thick] (0,-0.5) -- (0,\valD+1) node[above] {$\mathrm{Im}(z)$};

    \draw[dotted] (\valA,\valC) -- (\valA,0) node[below] {$a$};
    \draw[dotted] (\valB,\valC) -- (\valB,0) node[below] {$b$};
    \draw[dotted] (\valA,\valC) -- (0,\valC) node[left] {$c$};
    \draw[dotted] (\valA,\valD) -- (0,\valD) node[left] {$d$};

    \begin{scope}
        \clip (\valA,\valC) rectangle (\valB,\valD);

        \draw[step=\stepGrid, black!30, thin] (\valA,\valC) grid (\valB,\valD);

    \end{scope}

    \draw[thick, black] (\valA,\valC) rectangle (\valB,\valD);
    \node[black, anchor=south east] at (\valB, \valD) {$[a,b] \times [c,d]$};

    \draw[fill=black!20, draw=black, thick, fill opacity=0.8, even odd rule]
        plot [smooth cycle, tension=0.7] coordinates {
            (2, 2) (2.5, 4) (4.5, 3.5) (5, 2.5) (3.5, 1.5)
        }
        (3, 3) circle (0.3cm)
        (4.2, 2.3) ellipse (0.4cm and 0.2cm);

        \draw[black, ultra thick, line join=round]
            (2.0, 1.5) -- 
            (3.5, 1.5) -- 
            (3.5, 2.0) -- 
            (5.0, 2.0) -- 
            (5.0, 3.0) -- 
            (4.5, 3.0) -- 
            (4.5, 4.0) -- 
            (5.5, 4.0) -- 
            (5.5, 4.5) -- 
            (3.5, 4.5) -- 
            (3.5, 4.0) -- 
            (2.5, 4.0) -- 
            (2.5, 4.5) -- 
            (1.5, 4.5) -- 
            (1.5, 2.5) -- 
            (2.0, 2.5) -- 
            cycle;        

        \draw[<-, black, thick] (5.5, 4.25) -- ++(0.7,0) node[right, font=\bfseries] {$\gamma$};

    \node
        at (3.9, 3.4) {$\sigma(T)$};

\end{tikzpicture}
\hspace{-1cm}\caption{An appropriate curve for $T$} \label{fig:grid_gamma_snake}
\end{figure}
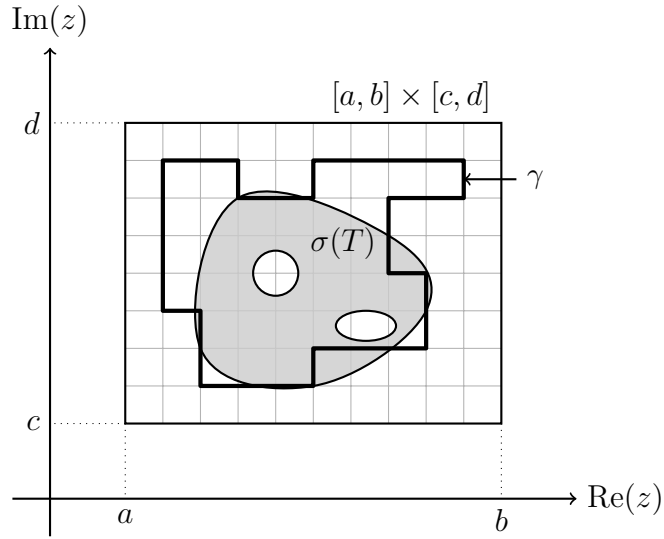

Assume $T$ is an operator as in Definition \ref{definition decomposability set}. Our goal is to exhibit a concrete subfamily $\F'_T$ of $\mathcal{F}_T$ that is an exhaustive cutting family for $T$. To establish this, it will suffice to show that $T$ admits  plain spectral cuts along every appropriate cycle in $\F'_T$. By Proposition \ref{proposicion union curvas}, this reduces to showing that $T$ admits a plain spectral cut along each of the appropriate curves that generate these cycles.

\medskip

Following the lines of \cite[Section 3]{GG3}, for each appropriate curve $\gamma$ we define
$$
X_\gamma(z) = \sumk (D_\Lambda-zI)^{-1/2}u_k\otimes e_k, \qquad z \in \gamma,
$$
$$
Y_\gamma(z) = \sumk e_k \otimes (D_\Lambda^*-\overline{z}I)^{-1/2}v_k, \qquad z \in \gamma.
$$
These are well-defined bounded operators on $H$. The proof of this fact is essentially the same as that of \cite[Proposition 3.3]{GG3}, taking into account that the curve considered there consists of a vertical segment attached to an arc of the unit circle $\T$. In our case, $\gamma$ is a finite union of horizontal and vertical segments, so the argument proceeds in the same way.	
	
\medskip
	
Again following \cite[Section 3]{GG3}, one can show that the operators $(I+Y_\gamma(z)X_\gamma(z))$ are invertible for each $z \in \gamma$, which leads to the representation
$$
(I+Y_\gamma(z)X_\gamma(z))^{-1} = \sumi \sumj a_{\gamma,i,j}(z)\, e_i\otimes e_j,
$$
where
$$
a_{\gamma,i,j}(z) = \pe{(I+Y_\gamma(z)X_\gamma(z))^{-1}e_j,e_i}.
$$

Moreover, defining
$$
f_T^{(i,j)}(z) = \sumn \frac{\alpha_n^{(i)}\overline{\beta_n^{(j)}}}{\lambda_n-z},
\qquad (i,j)\in \mathbb{N}\times\mathbb{N},
$$
the following result holds:	

\begin{theorem}
Let $T$ be an operator as in Definition \ref{definition decomposability set} and let $\gamma$ be an appropriate curve for $T$. Then there exists a constant $C_\gamma > 0$ such that, for every $x=\sumn x_n e_n \in H$,
$$
\sumi \left| \sumj x_j a_{\gamma,i,j}(z) \right|^2 \le C_\gamma \|x\|^2 .
$$
Moreover,
$$
\sumk \sumj x_j a_{\gamma,k,j}(z)\bigl(\delta_{k,n}+f_T^{(k,n)}(z)\bigr) = x_n,
$$
for each $n\in \mathbb{N}$.	
\end{theorem}

The proof runs parallel to that of \cite[Theorem 3.6]{GG3}, and we omit it.

\medskip

We are now in a position to introduce the operators $P_\gamma$, which will be shown below to be the idempotents associated with the spectral cuts along $\gamma$. For this purpose, given $\Lambda=(\lambda_n)_{n\in\mathbb{N}}$ and $A\subset\mathbb{C}$, denote
$$
N_A=\{n\in\mathbb{N}:\lambda_n\in A\}.
$$

\begin{theorem}\label{teorema definición f(T)}
Let $T$ be an operator as in Definition \ref{definition decomposability set}, and let $\gamma$ be an appropriate curve for $T$. If  $F=\overline{\inte(\gamma)}$ and $x=\sumn x_n e_n \in H$, the operator $P_\gamma$ given by
\begin{equation}
P_\gamma x = \sum_{n\in N_F} x_n e_n
+\sumn \sum_{k=1}^{\infty}
\left(
\frac{1}{2\pi i}\int_{\gamma}
\frac{\displaystyle \sumj \frac{x_j}{\lambda_j-\xi}
\left(\summ \overline{\beta_j^{(m)}}\, a_{\gamma,k,m}(\xi)\right)}
{\lambda_n-\xi}\, d\xi
\right)
\alpha_n^{(k)} e_n,\label{operador P}
\end{equation}
is well defined and bounded on $H$.	Moreover,
$$\ran(P_\gamma)\subseteq H_T(\overline{\inte(\gamma)}), \qquad \ker(P_\gamma)\subseteq \Hext.$$
\end{theorem}

\begin{proof}
In order to prove the boundedness of $P_\gamma$, it suffices to show that
		\begin{equation}\label{acotacion f(T)}
			\sumn \left| \sum_{k=1}^\infty \left( \frac{1}{2\pi i} \int_{\gamma}  \frac{ \sumj \frac{x_j}{\lambda_j-\xi}\left(\summ \overline{\beta_j^{(m)}}a_{\gamma,k,m}(\xi)\right)}{\lambda_n-\xi}        d\xi    \right)\al_n^{(k)}\right|^2 \leq C \norm{x}^2
		\end{equation}
		for some $C>0.$ The left-hand side of \eqref{acotacion f(T)} is smaller or equal than
$$ \sumn \left(\sum_{k=1}^\infty  \sumj |x_j| \int_{\gamma} \frac{ \left|\summ \overline{\beta_j^{(m)}}a_{\gamma,k,m}(\xi)\right|}{|\lambda_j-\xi||\lambda_n-\xi|}|d\xi|     \left|\al_n^{(k)}\right|\right)^2.$$
Recalling that $\gamma$ is a finite union of horizontal and vertical segments $s_\ell$, $\ell=1,\ldots,n$, whose endpoints have real and imaginary parts lying in $\Delta(T)$, and arguing as in the proof of \cite[Theorem 4.1]{GG3}, it follows that there exist constants $C_\ell>0$ such that
$$ \sumn \left(\sum_{k=1}^\infty  \sumj |x_j| \int_{s_\ell} \frac{ \left|\summ \overline{\beta_j^{(m)}}a_{\gamma,k,m}(\xi)\right|}{|\lambda_j-\xi||\lambda_n-\xi|}|d\xi|     \left|\al_n^{(k)}\right|\right)^2 \leq C_\ell\norm{x}^2$$ for $\ell=1,\cdots,n.$
Indeed, the proof in \cite{GG3} is carried out for vertical lines with endpoints outside $\sigma(T)$, but the same argument applies to shorter segments as well as to horizontal ones.

\noindent Finally, taking $C=\max\{C_\ell:\ell=1,\ldots,n\}$, we obtain \eqref{acotacion f(T)}, which yields the boundedness of $P_\gamma$.

\smallskip

Let us now prove that
$$\ran(P_\gamma) \subseteq H_T(\overline{\inte(\gamma)})
\quad \text{and} \quad
\ker(P_\gamma)\subseteq H_T(\overline{\exte(\gamma)}).
$$
For the containment $\ran(P_\gamma)\subseteq \Hint$, we apply the characterization of the spectral subspaces given in \cite[Theorem 2.2]{GG3}. Let $x=\sumn x_ne_n \in H$ and consider
$$y= P_\gamma x =\sum_{n\in N_F} x_ne_n+\sumn \sum_{k=1}^\infty \left( \frac{1}{2\pi i} \int_{\gamma} \frac{ \sumj \frac{x_j}{\lambda_j-\xi}\left(\summ \overline{\beta_j^{(m)}}a_{\gamma,k,m}(\xi)\right)}{\lambda_n-\xi}        d\xi    \right)\al_n^{(k)}e_n. $$
To prove that $y\in H_T(F)$, let us define for each $k\in \N$
		$$
        g_k(z) = \frac{1}{2\pi i} \int_{\gamma} \frac{ \sumj \frac{y_j}{\lambda_j-\xi}\left(\summ \overline{\beta_j^{(m)}}a_{\gamma, k,m}(\xi)\right)}{z-\xi} d\xi \qquad (z\in \C\setminus \gamma).
        $$
Arguing as in the proof of \cite[Theorem 4.2]{GG3}, one deduces  that each $g_k$ is a holomorphic function on $\C\setminus \gamma.$ Note that conditions $(i)$ and $(ii)$ in \cite[Theorem 2.2]{GG3} are also satisfied and to check $(iii)$, reasoning as in the proof of \cite[Theorem 4.2]{GG3}, it should be shown that
		$$ \sum_{n\in N_F} \frac{x_n\overline{\beta_n^{(k)}}}{\lambda_n-z} = \sumj x_j\overline{\beta_j^{(k)}} \frac{1}{2\pi i}\int_{\gamma}
		\frac{d\xi}{(\lambda_j-\xi)(z-\xi)}$$ for every $z\in \C\setminus F.$ But, at this point, it suffices to apply  Cauchy's integral formula to check that the  equality holds and hence, the containment $\ran(P_\gamma)\subseteq \Hint$ follows.

\medskip

Finally, to prove $\ker(P_\gamma)\subseteq H_T(\overline{\exte(\gamma)})$, we argue following \cite[Lemma 4.3]{GG3} and deduce that
        $$(I-P_\gamma)x = \sum_{n\in N_{F^c}} x_ne_n+\sumn \sum_{k=1}^\infty \left( \frac{1}{2\pi i} \int_{\tau} \frac{ \sumj \frac{x_j}{\lambda_j-\xi}\left(\summ \overline{\beta_j^{(m)}}a_{\gamma,k,m}(\xi)\right)}{\lambda_n-\xi}        d\xi    \right)\al_n^{(k)}e_n,$$
        where $\tau$ is the admissible cycle composed of $\partial ([a,b]\times [c,d])$ positively oriented and $\gamma$ taken with negative orientation. From here, it follows that
        $$
        \ran(I-P_\gamma) = \ker(P_\gamma) \subseteq \Hext,
        $$
        which completes the proof.
\end{proof}

	\begin{remark}\label{remark 5.6}
	    Note that if $x=\sumn x_ne_n$, the arguments in the proof of Theorem \ref{teorema definición f(T)} yield
        $$(I-P_\gamma)x = \sum_{n\in N_{F^c}} x_ne_n+\sumn \sum_{k=1}^\infty \left( \frac{1}{2\pi i} \int_{\tau} \frac{ \sumj \frac{x_j}{\lambda_j-\xi}\left(\summ \overline{\beta_j^{(m)}}a_{\gamma,k,m}(\xi)\right)}{\lambda_n-\xi}        d\xi    \right)\al_n^{(k)}e_n.$$
	\end{remark}

	\begin{proposition}\label{proposition idempotent producto}
Let $T$ be an operator as in Definition \ref{definition decomposability set}. There exists a measurable set $\tilde{\Delta}(T)\subseteq \Delta(T)$ with $m(\Delta(T)\setminus \tilde{\Delta}(T))=0$ such that, for every admissible grid $G\subseteq \Delta'(T)\times \Delta'(T)$ and every appropriate curve $\gamma$ whose forming segments have endpoints in $G$, the associated operator $P_\gamma$, defined in \ref{operador P}, satisfies
$$
P_{\gamma}(H_T(\overline{\exte(\gamma)}))=(I-P_\gamma)(H_T(\overline{\inte(\gamma)}))=\{0\}.
$$
Consequently, $P_\gamma$ is an idempotent.	
\end{proposition}

\begin{proof}
	We will show that $P_\gamma(H_T(\overline{\exte(\gamma)}))=\{0\}$; the proof for the identity $(I-P_\gamma)(H_T(\overline{\inte(\gamma)}))=\{0\}$ is similar. First, following \cite[Lemma 2.14]{GG2} and \cite[Lemma 4.4]{GG3}, one may check that $P_{\gamma}(T) \in \conm{T}.$

\medskip

Now, let $(\gamma_n)_{n\in \N}$ be a sequence of appropriate curves for $T$ such that for each $n\in \N$
    \begin{equation}\label{propiedad curvas 1}
        \overline{\inte(\gamma_{n})}\cap\sigma(T) \subset \inte(\gamma_{n+1})\cap\sigma(T), \qquad \overline{\inte(\gamma_n)}\cap\sigma(T) \subsetneq \inte(\gamma)\cap\sigma(T)
    \end{equation}
    and
    \begin{equation}\label{propiedad curvas 2}
        \inte(\gamma)=\bigcup_{n=1}^\infty \inte(\gamma_n).
    \end{equation}
		
\medskip
		
We claim that $P_{\gamma_n}(T)(H_T( \overline{\exte(\gamma)}) = 0$ for every $n \in \N.$  Indeed, by  Theorem \ref{teorema definición f(T)}, $\ran(P_{\gamma_n})\subseteq H_T(\overline{\inte(\gamma_n)})$. Since $P_{\gamma_n}$ commutes with $T$, \cite[Proposition 1.2.16]{LN00} yields that $\Hext$ is invariant under $P_{\gamma_n}$, so $P_{\gamma_n}(\Hext)\subseteq \Hext.$ Thus, $P_{\gamma_n}(\Hext) \subseteq \Hext\cap H_T(\overline{\inte(\gamma_n)})=\{0\}$, as claimed.

\medskip

Therefore, to prove that $P_\gamma(\Hext)=\{0\}$, we construct a set $\tilde{\Delta}(T)$ such that for every appropriate curve $\gamma$ with endpoints in a grid $G \subset \Delta'(T)\times\Delta'(T)$, there exists a sequence of appropriate curves $(\gamma_n)_{n\in\mathbb{N}}$ satisfying \eqref{propiedad curvas 1} and \eqref{propiedad curvas 2}, and such that
$$
P_{\gamma_n} \to P_\gamma \quad \text{as } n\to\infty
$$
in the weak operator topology. This fact will imply that $P_\gamma(\Hext) = \{0\}$, as desired.

\medskip

For this purpose, let $p\in \mathbb{N}$. By Egorov's theorem, there exists a measurable subset $\mathcal{X}_p\subset \Delta(T)$ such that $m(\Delta(T)\setminus \mathcal{X}_p) < 1/p$ and
$$
\sum_n \sum_k
\left(|\alpha_n^{(k)}|^2+|\beta_n^{(k)}|^2\right)
\left(
\frac{1}{|\PR(\lambda_n)-\xi|}
+\frac{1}{|\Impart(\lambda_n)-\xi|}
\right)
$$
converges uniformly on $\mathcal{X}_p$. Observe that the sets $\mathcal{X}_p$ contains, at most, a countable number of isolated points. If we denote by $\mathcal{X}_p'$ the associated derived set, let us define
$$
\tilde{\Delta}(T) = \bigcup_{p\in \N} \mathcal{X}'_p,
$$
and note that $m(\Delta(T)\setminus \tilde{\Delta}(T)) = 0$.

Let $p_0 \in \N$ and consider $\gamma$ to be an appropriate curve with the end-points of its forming segments lying in $\mathcal{X}'_{p_0}\times \mathcal{X}'_{p_0}$. It is clear that there exists a sequence of appropriate curves $(\gamma_n)_{n\in \N}$ with end-points of its forming segments in $\mathcal{X}'_{p_0}\times \mathcal{X}'_{p_0}$ satisfying \eqref{propiedad curvas 1} and \eqref{propiedad curvas 2}. At this point, to prove that $P_{\gamma_n}$ converges to $P_\gamma$ in the weak operator topology, it is enough to follow the same arguments exposed in the proof of \cite[Lemma 4.4]{GG3}, with the obvious modifications.

\medskip

Finally, by Theorem \ref{teorema definición f(T)},
$$\ran(P_\gamma(I-P_\gamma)) \subseteq P_\gamma(\Hext) = \{0\},$$
so $P_\gamma$ is an idempotent, which yields Proposition \ref{proposition idempotent producto}.
\end{proof}

With Proposition \ref{proposition idempotent producto} at hand, denote by $\F'_T$ the subfamily of $\F_T$ consisting of appropriate cycles for $T$ given by finite unions of curves of the form $\partial([a,b]\times[c,d])$ together with appropriate curves $\gamma$ whose segments have endpoints in $\tilde{\Delta}(T)\times \tilde{\Delta}(T)$.

\medskip

The following theorem is the main result of this section; it builds on the previous results and extends the decomposability established in \cite{GG3} to super-decomposability:

\begin{theorem}
Let $T$ be an operator as in Definition \ref{definition decomposability set}, and let $G\subset \tilde{\Delta}(T)\times \tilde{\Delta}(T)$ be an admissible grid. For every appropriate curve $\gamma$ whose segments have endpoints in $G$,
\begin{equation}\label{identidad final}
    \ran(P_\gamma) = \Hint, \qquad \ker(P_\gamma) = \Hext.
\end{equation}
Consequently, $\F_T'$ is an exhaustive cutting family for $T$, and hence $T$ is super-decomposable.
\end{theorem}

\begin{proof}
Note that $\F_T'$ satisfies condition $(ii)$ of Definition \ref{definition exhaustive}. To verify condition $(i)$, it suffices to prove \eqref{identidad final}. We establish the first identity; the proof of the second is analogous.

By Theorem \ref{teorema definición f(T)}, $\ran(P_\gamma)\subseteq \Hint$. Let $x\in \Hint$ and write $x=P_\gamma x+(I-P_\gamma)x$. Since $\Hint$ is invariant under $P_\gamma$, we have $(I-P_\gamma)x\in (I-P_\gamma)\Hext=\{0\}$ by Proposition \ref{proposition idempotent producto}. Hence \eqref{identidad final} holds, and $\ran(P_\gamma)=\Hint$.

Therefore, $\F_T'$ is an exhaustive cutting family for $T$. By Theorem \ref{teorema super-decomposable}, $T$ is super-decomposable.
\end{proof}

Finally, we derive an explicit integral representation for the unconventional functional calculus for $T$ associated with $\mathcal{F}'_T$:

\begin{theorem}
Let $T$ be an operator as in Definition \ref{definition decomposability set},  $G\subset \tilde{\Delta}(T)\times \tilde{\Delta}(T)$ an admissible grid and $\gamma$ an appropriate curve for $T$ whose segments have endpoints in $G$. Let $F$ denote $\overline{\inte(\gamma)}$ and let $U$ be an open set containing $F$. If $f: U \rightarrow \mathbb{C}$ is a holomorphic function, for every $x=\sumn x_ne_n \in H$,
     $$
     f_\gamma(T)x = \sum_{n\in N_F} f(\lambda_n)x_ne_n + \sumn \sum_{k=1}^\infty \left( \frac{1}{2\pi i} \int_{\gamma} f(\xi) \frac{ \sumj \frac{x_j}{\lambda_j-\xi}\left(\summ \overline{\beta_j^{(m)}}a_{\gamma,k,m}(\xi)\right)}{\lambda_n-\xi}        d\xi    \right)\al_n^{(k)}e_n.
	$$
\end{theorem}

\begin{proof}
    Recall that $f_\gamma(T) =  f(T\mid_{\ran(P_\gamma)})P_\gamma$. As $P_\gamma$ is the projection associated with the plain spectral cut along $\gamma$, it follows that $P_\gamma \in \biconm{T}.$ Denote by $A$ the operator given by
    $$
    A x= \sum_{n\in N_F} f(\lambda_n)x_ne_n + \sumn \sum_{k=1}^\infty \left( \frac{1}{2\pi i} \int_{\gamma} f(\xi) \frac{ \sumj \frac{x_j}{\lambda_j-\xi}\left(\summ \overline{\beta_j^{(m)}}a_{\gamma,k,m}(\xi)\right)}{\lambda_n-\xi}        d\xi    \right)\al_n^{(k)}e_n,
    $$
    for $x=\sumn x_ne_n \in H$.

    As in Theorem \ref{teorema definición f(T)}, $A$ is a bounded operator on $H$ commuting with $T$ and satisfying $\ran(A)\subseteq \Hint$. Hence $AP_\gamma=P_\gamma A=A$. Let $\tau$ be the admissible cycle consisting of $\partial([a,b]\times[c,d])$, positively oriented, together with $\gamma$ with negative orientation, and define
    $$
    B x= \sum_{n\in N_{F^c}} f(\lambda_n)x_ne_n + \sumn \sum_{k=1}^\infty \left( \frac{1}{2\pi i} \int_{\tau} f(\xi) \frac{ \sumj \frac{x_j}{\lambda_j-\xi}\left(\summ \overline{\beta_j^{(m)}}a_{\gamma,k,m}(\xi)\right)}{\lambda_n-\xi}        d\xi    \right)\al_n^{(k)}e_n,
    $$
for $x=\sumn x_ne_n \in H$. Again, $B$ is a bounded operator commuting with $T$ and satisfying $\ran(B)\subseteq \Hext$. Hence $BP_\gamma=P_\gamma B=0$. Moreover, since
    $$
    f(T\mid_{\ran(P_\gamma)})x = \frac{1}{2\pi i} \int_{\partial ([a,b]\times [c,d])}f(z) (z I-T)^{-1}x dz,
    $$
    for every $x\in H$, it follows that $(A+B)\mid_{\ran(P_\gamma)}=f(T\mid_{\ran(P_\gamma)})$, due to the expression for $(zI-T)^{-1}$ obtained in the proof of \cite[Lemma 4.3]{GG3}.

    Finally, we have,
    $$
    f_\gamma(T) = f_\gamma(T\mid_{\ran(P_\gamma)}) P_\gamma = (A+B)\mid_{\ran(P_\gamma)} P_\gamma = AP_\gamma + BP_\gamma  = A,
    $$
    which concludes the proof.
\end{proof}

\end{document}